\documentclass[11pt]{amsart}

\usepackage[a4paper,hmargin=3.5cm,vmargin=3.5cm]{geometry}
\usepackage{amsfonts,amssymb,amscd,amstext}
\usepackage{graphicx}
\usepackage[dvips]{epsfig}

\usepackage{fancyhdr}
\usepackage{times}

\usepackage{enumerate}
\usepackage{titlesec}
\usepackage{mathrsfs}

\def\Ccal{\mathcal{C}}
\def\Acal{\mathcal{A}}
\def\Ncal{\mathcal{N}}
\def\Mcal{\mathcal{M}}

\def\Hcal{\mathcal{H}}
\def\Kcal{\mathcal{K}}
\def\Bcal{\mathcal{B}}
\def\Fcal{\mathcal{F}}
\def\Gcal{\mathcal{G}}
\def\Pcal{\mathcal{P}}
\def\Ocal{\mathcal{O}}
\def\Qcal{\mathcal{Q}}

\def\Ucal{\mathcal{U}}

\def\Scal{\mathcal{S}}

\def\Cscr{\mathscr{C}}
\def\I{\mathfrak{I}}
\def\J{\mathfrak{J}}

\def\c{\mathbb{C}}
\def\r{\mathbb{R}}
\def\n{\mathbb{N}}
\def\z{\mathbb{Z}}
\def\k{\mathbb{K}}
\def\s{\mathbb{S}}
\def\p{\mathbb{P}}

\def\Agot{\mathfrak{A}}
\def\hgot{\mathfrak{h}}
\def\mgot{\mathfrak{m}}
\def\ggot{\mathfrak{g}}
\def\igot{\mathfrak{i}}
\def\jgot{\mathfrak{j}}
\def\pgot{\mathfrak{p}}
\def\div{\mathfrak{Div}}

\newcommand{\sub}[1]{\text{\b{$#1$}}}

\titleformat{\section}
{\filcenter\bfseries\large} {\thesection{.}}{0.2cm}{}
\titleformat{\subsection}[runin]
{\bfseries} {\thesubsection{.}}{0.15cm}{}[.]
\titleformat{\subsubsection}[runin]
{\em}{\thesubsubsection{.}}{0.15cm}{}[.]

\usepackage[up,bf]{caption}

\newtheorem{theorem}{Theorem}[section]
\newtheorem{proposition}[theorem]{Proposition}
\newtheorem{claim}[theorem]{Claim}
\newtheorem{lemma}[theorem]{Lemma}
\newtheorem{corollary}[theorem]{Corollary}
\newtheorem{remark}[theorem]{Remark}
\newtheorem{definition}[theorem]{Definition}

\theoremstyle{definition}

\numberwithin{equation}{section}
\numberwithin{figure}{section}

\usepackage{color}

\begin{document}

\fancyhead[LO]{Approximation theory for non-orientable minimal surfaces} 
\fancyhead[RE]{A. Alarc\'{o}n$\,$  and$\,$ F.J. L\'{o}pez} 
\fancyhead[RO,LE]{\thepage} 

\thispagestyle{empty}

\thispagestyle{empty}

\vspace*{1cm}
\begin{center}
{\bf\LARGE Approximation theory for non-orientable minimal surfaces and applications}

\vspace*{0.5cm}

{\large\bf Antonio Alarc\'{o}n$\;$ and$\;$ Francisco J.\ L\'{o}pez}

\end{center}

\footnote[0]{\vspace*{-0.4cm}

\noindent A.\ Alarc\'{o}n,  F.\ J.\ L\'{o}pez

\noindent Departamento de Geometr\'{\i}a y Topolog\'{\i}a, Universidad de Granada, E-18071 Granada, Spain.

\noindent e-mail: {\tt alarcon@ugr.es}, {\tt fjlopez@ugr.es}
}


\begin{quote}
{\small
\noindent {\bf Abstract}\hspace*{0.1cm} We prove a version of the classical Runge and Mergelyan uniform approximation theorems for non-orientable minimal surfaces in Euclidean $3$-space $\r^3.$ Then, we obtain some geometric applications. Among them, we emphasize the following ones:
\begin{itemize}
\item A Gunning-Narasimhan type theorem for non-orientable conformal surfaces.
\item An existence theorem for non-orientable minimal surfaces in $\r^3,$ with arbitrary conformal structure, properly projecting into a plane.
\item An existence result for non-orientable minimal surfaces in $\r^3$ with arbitrary conformal structure and Gauss map omitting one projective direction.
\end{itemize}

\vspace*{0.1cm}

\noindent{\bf Keywords}\hspace*{0.1cm} Uniform approximation, non-orientable minimal surfaces.

\vspace*{0.1cm}

\noindent{\bf Mathematics Subject Classification (2010)}\hspace*{0.1cm} 49Q05, 30E10.
}
\end{quote}


\section{Introduction}\label{sec:intro}

The Runge and Mergelyan theorems are the central results in the theory of uniform approximation by holomorphic functions in one complex variable. The former, which dates back to 1885, asserts that if the complement $\c\setminus K$ of a compact set $K\subset\c$ has no relatively compact connected components, then every holomorphic function in (an open neighborhood of) $K$ can be approximated, uniformly on $K$, by holomorphic  functions on $\c$; cf. \cite{Runge}. If $K\subset\c$ is an arbitrary compact set, then Mergelyan's theorem \cite{Mergelyan}, which dates back to 1951, ensures that continuous functions $K\to \c$, holomorphic in the interior $K^\circ$ of $K$, can be approximated uniformly on $K$ by holomorphic functions in open neighborhoods of $K$ in $\c$.
In 1958 Bishop \cite{Bishop-RM} extended these results to Riemann surfaces (see \cite{Forster} for a modern proof using functional analysis):

\noindent {\bf Runge-Mergelyan's Theorem.} 
{\em Let $\Ncal$ be an open Riemann surface (i.e., non-compact) and let $K\subset \Ncal$ be a compact set such that $\Ncal\setminus K$ has no relatively compact connected components. Then any continuous function $K\to\c$, holomorphic in the interior $K^\circ$ of $K$, can be uniformly approximated on $K$ by holomorphic functions $\Ncal\to\c$.}

A compact subset $K\subset \Ncal$ satisfying the hypothesis of the above theorem is said to be a {\em Runge set} in the open Riemann surface $\Ncal$. Runge and Mergelyan's theorems admit plenty of generalizations; the extension of Runge's theorem to functions of several complex variables is known as the Oka-Weil theorem (see e.g.\ \cite{Hormander-SCV}), and in the more general setting, maps $S\supset K\to O$ from a holomorphically convex set $K$ of a Stein manifold $S$ to an Oka manifold $O$ satisfy the Runge property (see \cite{Forstneric-book} for a good reference).

On the other hand, conformal minimal immersions of open Riemann surfaces into Euclidean space are harmonic functions. This basic fact has strongly influenced the theory of minimal surfaces, furnishing it of powerful tools coming from Complex Analysis. In particular, Runge's theorem (combined with the L\'opez-Ros transformation for minimal surfaces; see \cite{LopezRos}) has been the key tool for constructing complete hyperbolic minimal surfaces in $\r^3$ of finite topology; see \cite{JorgeXavier,Nadirashvili,Morales} for pioneering papers.
However, the direct application of Runge's theorem has a limited reach, and it seems to be insufficient for constructing minimal surfaces with more complicated geometry. With the aim of overcoming this constraint, the authors \cite[Theorem 4.9]{AL-proper} obtained a Runge-Mergelyan type theorem for conformal minimal immersions of open Riemann surfaces into $\r^3.$ This result has been a versatile tool for constructing both minimal surfaces in $\r^3$ and null holomorphic curves in the Complex $3$-space $\c^3$; see \cite{AL-proper,AFL-1,AL-CY,AL-Israel,AL-conjugada} for a number of applications. For instance, and in contrast to the severe restrictions imposed by the use of the L\'opez-Ros transformation, it allows one to prescribe the conformal structure of the examples.

In the same spirit, a Runge-Mergelyan type theorem for a large family of directed holomorphic immersions of open Riemann surfaces into $\c^n$ (including null curves), $n\geq 3$,  has been recently shown, with different techniques, by Alarc\'on and Forstneri\v c \cite{AF-2}.

In this paper we focus on non-orientable minimal surfaces in $\r^3$. This subject should not be considered as a minor or secondary one; on the contrary, non-orientable surfaces present themselves quite naturally in the origin itself of minimal surface theory (recall for instance that a M\"obius minimal strip can be obtained by solving a simple Plateau problem; see \cite{Douglas-non} for more information), and they present a rich and interesting geometry. Non-orientable minimal surfaces were first studied systematically by Lie \cite{Lie-non} in the third quarter of the 19th century; the development of their global theory was began by Meeks \cite{Meeks-8pi}. A particular issue is that constructing non-orientable surfaces via Weierstrass representation is in general hard, due to the higher subtlety of the period problem. Runge's theorem has been already used ad hoc in several constructions of complete non-orientable minimal surfaces in $\r^3$ (see \cite{Lopez-slab,LopezMartinMorales-non,FerrerMartinMeeks}); however, as in the orientable case, its direct use seems to be not enough for more involved constructions. 

The aim of this paper is to prove a Runge-Mergelyan type theorem for non-orientable minimal surfaces in $\r^3$. For a precise statement of our main result, the following notation is required. Every non-orientable minimal surface $M\subset \r^3$ can be represented by a triple $(\Ncal,\I,X)$, where $\Ncal$ is an open Riemann surface, $\I\colon \Ncal\to\Ncal$ is an antiholomorphic involution without fixed points, and $X\colon \Ncal\to\r^3$ is an {\em $\I$-invariant} conformal minimal immersion (that is, satisfying $X\circ\I=X$) such that $M=X(\Ncal)$; see \cite{Meeks-8pi} and Subsec.\ \ref{sec:minimal} for details. We say that a subset $S$ of $\Ncal$ is {\em $\I$-admissible} (see Def.\ \ref{def:admi}) if it is Runge in $\Ncal$, $\I(S)=S$, and $S=R_S\cup C_S$, where $R_S:=\overline{S^\circ}$ consists of a finite collection of pairwise disjoint compact regions in $\Ncal$ with $\Ccal^1$ boundary, and $C_S:=\overline{S\setminus R_S}$ is a finite collection of pairwise disjoint analytical Jordan arcs, meeting $R_S$ only in their endpoints, and such that their intersections with the boundary $b R_S$ of $R_S$ are transverse. Finally, we say that a $\Ccal^1$ map $Y\colon S\to\r^3$ is an {\em $\I$-invariant generalized minimal immersion} (see Subsec.\ \ref{sec:minimal-I}) if $Y|_{R_S}$ is a conformal minimal immersion, $Y|_{C_S}$ is regular, and $Y\circ\I=Y$. 

Our main result asserts that:
\begin{theorem}\label{th:intro}
Let $\Ncal$ be an open Riemann surface, let $\I\colon \Ncal\to\Ncal$ be an antiholomorphic involution without fixed points, and let $S\subset\Ncal$ be an $\I$-admissible set.

Then any $\I$-invariant generalized minimal immersion $S\to\r^3$ can be $\Ccal^1$ uniformly approximated on $S$ by $\I$-invariant conformal minimal immersions $\Ncal\to\r^3$.
\end{theorem}

Concerning the proof, it is important to point out that the compatibility condition with respect to the antiholomorphic involution and the higher difficulty of the period problem require a much more involved and careful analysis than in the orientable case; cf.\ \cite{AL-proper}.

Theorem \ref{th:intro} has many geometric applications. In Theorem \ref{th:proper} we show that, for any  open Riemann surface $\Ncal$ and antiholomorphic involution $\I\colon\Ncal\to\Ncal$ without fixed points, there exist $\I$-invariant conformal minimal immersions $\Ncal\to\r^3$ properly projecting into a plane (cf. \cite{AL-proper} for the analogous result in the orientable case). This links with an old question by Schoen and Yau \cite[p.\ 18]{SchoenYau-harmonic}; see \cite{AL-proper,AL-conjugada} for a good reference.  We also prove an existence theorem of complete conformal $\I$-invariant minimal immersions $\Ncal\to\r^3$ with a prescribed coordinate function; see Theorem \ref{th:coordenada}. As a consequence, in Corollary \ref{co:coordenada} we exhibit complete non-orientable minimal surfaces in $\r^3$ whose Gauss map omits one point of the projective plane $\r\p^2$ (see \cite{AFL-1} for the orientable framework).  Other geometric applications of Theorem \ref{th:intro} will be obtained in the forthcoming paper \cite{AL-non1}.

Theorem \ref{th:intro} follows from the more general Theorem \ref{th:Mergelyan}, which also deals with the flux map of the approximating surfaces. In particular, Theorem \ref{th:Mergelyan} implies the analogous result of Theorem \ref{th:intro} for null holomorphic curves $F\colon \Ncal\to\c^3$ enjoying the symmetry $F\circ\I=\overline F$; see Corollary \ref{co:null}.  
We also derive a Runge-Mergelyan type theorem for harmonic functions $h\colon \Ncal\to\r$ satisfying $h\circ\I=h$ (see Theorem \ref{th:harmonic}). 

Finally, in a different line of applications, we prove an extension of the classical Gunning-Narasimhan theorem \cite{GunningNarasimhan} (see also \cite{KusunokiSainouchi}); more specifically, we show that, for any open Riemann surface $\Ncal$ and any antiholomorphic involution $\I\colon \Ncal\to\Ncal$ without fixed points, there exist holomorphic $1$-forms $\vartheta$ on $\Ncal$ with $\I^*\vartheta=\overline{\vartheta}$ and prescribed periods and canonical divisor (see Theorem \ref{th:GN}).

\medskip

\noindent{\bf Outline of the paper.} The necessary notation and background on non-orientable minimal surfaces in $\r^3$ is introduced in Sec.\ \ref{sec:prelim}. In Sec.\ \ref{sec:admi} we describe the compact subsets involved in the Mergelyan type approximation, and define the notion of conformal non-orientable minimal immersion from such a subset into $\r^3$. In Sec.\ \ref{sec:approx} we prove several preliminary approximation results that flatten the way to the proof of the main theorem in Sec.\ \ref{sec:Mer}. Finally, the applications are derived in Sec.\ \ref{sec:apli}.


\section{Preliminaries}\label{sec:prelim}

Let $\|\cdot\|$ denote the Euclidean norm in $\k^n$ ($\k=\r$ or $\c$). Given a compact topological space $K$ and a continuous map $f\colon K\to\mathbb{K}^n,$ we denote by 
\[
\|f\|_{0,K}:= \max_K \big\{ \|f(p)\|\colon p\in K\}
\]
the maximum norm of $f$ on $K.$ The corresponding space of continuous functions on $K$ will be endowed with the $\Ccal^0$ topology associated to $\|\cdot\|_{0,K}.$

Given a topological surface $N,$ we denote by $b N$ the (possibly non-connected) $1$-dimensional topological manifold determined by its boundary points. Given a subset $A \subset N,$ we denote by $A^\circ$ and $\overline{A}$ the interior and the closure of $A$ in $N$, respectively. Open connected subsets of $N\setminus b N$ will be called {\em domains} of $N$, and those proper connected topological subspaces of $N$ being compact surfaces with boundary will said to be  {\em regions} of $N$.


\subsection{Riemann surfaces and non-orientability}\label{sec:riemann}

It is well known that any Riemann surface is orientable; in fact, the conformal structure of the surface induces a (positive) orientation on it. In this subsection, we describe the notion of {\em non-orientable Riemann surface}.

A Riemann surface $\Ncal$ is said to be {\em open} if it is non-compact and $b \Ncal =\emptyset.$ We denote by $\partial$ the global complex operator given by $\partial|_U=\frac{\partial}{\partial z} dz$ for any conformal chart $(U,z)$ on $\Ncal.$ 

\begin{definition}
Let $\sub{\Ncal}$ be a smooth non-orientable surface with empty boundary. A system of coordinates $\Cscr$ on $\sub{\Ncal}$ is said to be a {\em conformal structure} on $\sub{\Ncal}$ if the change of coordinates is conformal or anticonformal. 
The couple $(\sub{\Ncal},\Cscr)$ is said to be a {\em non-orientable Riemann surface}. If there is no place for ambiguity, we simply write $\sub{\Ncal}$ instead of $(\sub{\Ncal},\Cscr).$
\end{definition}

\begin{definition}\label{def:non}
Let $\sub{\Ncal}\equiv (\sub{\Ncal},\Cscr)$ be a non-orientable Riemann surface. Denote by $\pi\colon \Ncal\to \sub{\Ncal}$ the oriented 2-sheeted covering of $\sub{\Ncal},$ and call $\I\colon \Ncal\to \Ncal$ the deck transformation of $\pi.$ Call  $\pi^*(\Cscr)$ the holomorphic system of coordinates in $\Ncal$ determined by the positively oriented lifts by $\pi$ of the charts in $\Cscr.$ 

Notice that the couple $\Ncal \equiv (\Ncal,\pi^*(\Cscr))$ is a (connected) open Riemann surface and  $\I$ is an antiholomorphic involution in $\Ncal$ without fixed points. The conformal map $\pi\colon \Ncal \to \sub{\Ncal}$ is said to be the conformal orientable two-sheeted covering of $\sub{\Ncal}.$

Objects related to $\sub{\Ncal}$ will be denoted with underlined text (for instance: $\sub{S},$ $\sub{X},$ etc.), whereas those related to $\Ncal$ will be not.
\end{definition}

As a consequence of Def.\ \ref{def:non}, the non-orientable Riemann surface $\sub{\Ncal}$ can be  naturally identified with the orbit space $\Ncal/\I,$ and the covering map $\pi$ with the natural projection $\Ncal \to \Ncal/\I.$ In other words, a non-orientable Riemann surface $\sub{\Ncal}$ is nothing but a connected open Riemann surface $\Ncal$ equipped with an antiholomorphic involution $\I$ without fixed points.

From now on in Section \ref{sec:prelim}, let $\sub{\Ncal},$ $\Ncal,$ $\pi,$ and $\I,$ be as in Def.\ \ref{def:non}.

\begin{definition}\label{def:saturated}
A subset $A\subset \Ncal$ is said to by $\I$-invariant if  $\I(A)=A,$  or equivalently, $\pi^{-1}(\pi(A))=A.$
If $A$ is $\I$-invariant,  we write  $\sub{A}=\pi(A).$ Likewise, given $\sub{B}\subset\sub{\Ncal},$ we write $B$ for the $\I$-invariant set $\pi^{-1}(\sub{B}).$
\end{definition}

\begin{definition}\label{def:invariant}
Let $A$ be an $\I$-invariant subset in $\Ncal$ and let $f\colon A\to\r^n$ be a map ($n\in\n$). The map $f$ is said to be {\em $\I$-invariant} if 
\[
f\circ (\I|_A)=f .
\]
In this case, we denote by $\sub{f}$ the only map $\sub{f}\colon \sub{A}\to\r^n$ satisfying $f=\sub{f}\circ(\pi|_A).$
Likewise, given a map $\sub{f}\colon \sub{A}\to\r^n$ we denote by $f$ the $\I$-invariant map $f=\sub{f}\circ(\pi|_A)\colon A\to\r^n.$
\end{definition}

For any set $A \subset \Ncal,$ we denote by $\div(A)$  the free commutative group of divisors of $A$ with multiplicative notation. If $D=\prod_{i=1}^n Q_i^{n_i} \in \div(A),$ where $n_i \in \z\setminus\{0\}$ for all $i\in\{1,\ldots,n\},$ we set ${\rm supp}(D):=\{Q_1,\ldots,Q_n\}$ the support of $D.$ 
A divisor $D \in \div(A)$ is said to be integral if $D=\prod_{i=1}^n Q_i^{n_i}$ and $n_i\geq 0$ for all $i.$ Given $D_1,$ $D_2 \in \div(A),$ $D_1 \geq D_2$ means that $D_1 D_2^{-1}$ is  integral.
If $A$ is $\I$-invariant, then we denote by $\div_\I(A)$ the group of $\I$-invariant divisors of $A$; that is to say, satisfying $\I(D)=D.$

In the sequel, $W$ will denote an $\I$-invariant open subset of $\Ncal.$ 

We denote by
\begin{itemize}
\item $\Fcal_{\hgot,\I}(W)$ the real vectorial space of holomorphic functions $f$ on $W$ such that $f\circ (\I|_W)=\overline{f},$ 
\item $\Fcal_{\mgot,\I}(W)$ the real vectorial space of meromorphic functions $f$ on $W$ such that $f\circ (\I|_W)=\overline{f},$
\item $\Omega_{\hgot,\I}(W)$ the real vectorial space of holomorphic $1$-forms $\theta$ on $W$ such that $\I^*(\theta)=\overline{\theta},$ and 
\item $\Omega_{\mgot,\I}(W)$ the real vectorial space of meromorphic $1$-forms $\theta$ on $W$ such that $\I^*(\theta)=\overline{\theta}$ 
\end{itemize}
(here, and from now on, $\bar{\cdot}$ means complex conjugation). We also denote by
\begin{itemize}
\item $\Gcal_\I(W)$ the family of meromorphic functions $g$ on $W$ satisfying $g\circ(\I|_W)=-1/\overline{g}.$
\end{itemize}
By elementary symmetrization arguments, it is easy to check that $\Fcal_{\hgot,\I}(W)$, $\Omega_{\hgot,\I}(W)\neq \emptyset$.
It is also known that $\Gcal_\I(\Ncal)\neq\emptyset$ when $\Ncal$ is a compact Riemann surface (see \cite{Martens}).  As application of Theorem \ref{th:Mergelyan}, we will prove that in fact every open non-orientable Riemann surface $(\Ncal,\I)$ carries conformal maps into the projective plane omitting one point (see Corollary \ref{co:coordenada}); in particular $\Gcal_\I(\Ncal)\neq\emptyset$.

Let us recall some well-known topological facts regarding non-orientable surfaces.

In the remaining of this subsection, we will assume that $W$ is a domain of finite topology. Then $(W,\I|_W)$ is topologically equivalent to $(\Scal\setminus \{P_1,\ldots,P_{k+1},\J(P_1),\ldots,\J(P_{k+1})\},\J)$, where  $\Scal$ is a compact 
surface  of genus $\nu$, $\J:\Scal \to \Scal$ is an orientation reversing involution without fixed points, and $\{P_1,\ldots,P_{k+1}\}\subset \Scal$, $k\in \n\cup\{0\}$.

As a consequence, the first homology groups $\Hcal_1(W,\z)$ and $\Hcal_1(W,\r)$ of $W$ are of dimension $2\nu_0+1$, where  $\nu_0:=\nu+k$. Furthermore, $\Hcal_1(W,\r)$ admits an $\I$-basis accordingly to the following definition:
\begin{definition} \label{def:ibasis}
A basis $\Bcal=\{c_0,c_1,\ldots,c_{\nu_0},d_1,\ldots,d_{\nu_0}\}$ of $\Hcal_1(W,\r)$ is said to be an $\I$-basis if 
\begin{itemize}
\item $c_j:=\gamma_j -\I_*(\gamma_j)$, $j=0,1,\ldots,\nu_0$, and
\item $d_j:=\gamma_j+\I_*(\gamma_j)$, $j=1,\ldots,\nu_0$,
\end{itemize}
for some closed curves $\{\gamma_j\colon j=0,\ldots,\nu_0\}\subset \Hcal_1(W,\z)$.
Observe that 
\begin{equation}\label{eq:c+d}
\text{$\I_*(c_j)=-c_j$ \quad and \quad $\I_*(d_j)=d_j$ \quad for all $j$.}
\end{equation}
\end{definition}

Let  $\Hcal^1_{{\rm hol,\I}}(W)$ be the first real De
Rham  cohomology group $\Omega_{\hgot,\I}(W)/\sim$, where as usual $\sim$ denotes the equivalence relation ``{\em the difference is exact}''. Notice that \eqref{eq:c+d} gives that $\Re \int_{c_j} \tau=0$ and $\Im \int_{d_j} \tau=0$ for all $j$ and  $\tau\in\Omega_{\hgot,\I}(W)$. Further, basic cohomology theory gives that  the map
\begin{equation}\label{eq:isom}
\text{$\Hcal^1_{{\rm hol},\I}(W) \to \r^{\nu_0+1}\times \r^{\nu_0}$, \quad $[\tau] \mapsto \Big( \big(-\imath \int_{c_j} \tau\big)_{j=0,\ldots,\nu_0}\,,\,  \big(\int_{d_j} \tau \big)_{j=1,\ldots,\nu_0}\Big)$,}
\end{equation}
is a (real) linear isomorphism for any $\I$-basis $\{c_0,c_1,\ldots,c_{\nu_0},d_1,\ldots,d_{\nu_0}\}$ of $\Hcal_1(W,\r)$.


\subsection{Non-orientable minimal surfaces} \label{sec:minimal}

In this subsection we describe the Weierstrass representation formula for non-orientable minimal surfaces, and introduce some notation.

\begin{definition}\label{def:M(A)}
A map $\sub{X}\colon \sub{\Ncal}\to\r^3$ is said to be a {\em conformal non-orientable minimal immersion} if the $\I$-invariant map
\[
X:=\sub{X}\circ\pi\colon \Ncal\to \r^3\enskip \text{(see Def.\ \ref{def:invariant})}
\]
is a conformal minimal immersion. In this case, $X(\Ncal)=\sub X(\sub \Ncal)\subset\r^3$ is a non-orientable minimal surface.

For any $\I$-invariant subset $A\subset \Ncal,$  we denote by $\Mcal_\I(A)$ the space of $\I$-invariant conformal minimal immersions of $\I$-invariant open subsets  of $\Ncal$ containing $A$ into $\r^3.$ 
\end{definition}

Let $A \subset \Ncal$ be an $\I$-invariant subset, and let $X\in \Mcal_\I(A)$. Denote by $\phi_j=\partial X_j,$ $j=1,2,3,$ and $\Phi=\partial X\equiv (\phi_j)_{j=1,2,3}.$  The $1$-forms  $\phi_j$ are holomorphic (on an open neighborhood of $A$),  have no real periods, and satisfy
\begin{equation}\label{eq:conformal}
\sum_{j=1}^3 \phi_j^2=0
\end{equation}
and
\begin{equation}\label{eq:Wdata-non}
\I^*\Phi=\overline{\Phi}
\end{equation}
(see \cite{Meeks-8pi}). The intrinsic metric in (an open neighborhood of) $A$ is given by
$ds^2=\sum_{j=1}^3 |\phi_j|^2;$ hence 
\begin{equation}\label{eq:immersion}
\text{$\sum_{j=1}^3 |\phi_j|^2$ vanishes nowhere on $A$}.
\end{equation}
By definition, the triple $\Phi$ is said to be the {\em Weierstrass representation} of $X.$ The meromorphic function \begin{equation}\label{eq:g}
g=\frac{\phi_3}{\phi_1- \imath \phi_2}
\end{equation}
(here, and from now on, we denote by $\imath=\sqrt{-1}$)
corresponds to the Gauss map of $X$ up to the stereographic projection, and 
\begin{equation}\label{eq:recoverX}
\Phi=\Big(\frac{1}{2}\big(\frac{1}{g}-g\big),\frac{\imath}{2}\big(\frac{1}{g}+g\big),1\Big)\phi_3 
\end{equation}
(see \cite{Osserman-book}). It follows from \eqref{eq:Wdata-non} and \eqref{eq:g} that the complex Gauss map $g\colon A\to\overline{\c}:=\c\cup\{\infty\}$ of $X$ satisfies that
\begin{equation}\label{eq:g-non}
g\circ(\I|_A) = -\frac{1}{\overline{g}}.
\end{equation}

\begin{remark}\label{rem:G}
Denote by $\Agot\colon \overline{\c}\to\overline{\c}$ the antipodal map $\Agot(z)=-1/\overline{z},$ by $\r\p^2=\overline{\c}/\Agot$ the projective plane, and by $\pi_\Agot\colon \overline{\c}\to\r\p^2$ the orientable 2-sheeted covering of $\r\p^2.$ 

Every meromorphic $g$ in  (an open neighborhood of) $A$ satisfying \eqref{eq:g-non}  induces a unique conformal map $G\colon \sub{A}\to\r\p^2$ such that $G\circ (\pi|_A)=\pi_\Agot\circ g.$  
\end{remark}

\begin{definition}\label{def:gauss-non}
The conformal map $G\colon \sub{A}\to\r\p^2$ induced by the complex Gauss map $g$ of $X$ is said to be the {\em complex Gauss map} of the conformal non-orientable minimal immersion $\sub{X}.$ 
\end{definition}

Conversely, any vectorial holomorphic $1$-form $\Phi=(\phi_j)_{j=1,2,3}$ on  (an open neighborhood of) $A$ {\em without real periods}, satisfying \eqref{eq:conformal}, \eqref{eq:Wdata-non}, and \eqref{eq:immersion}, determines an $\I$-invariant conformal minimal immersion $X\colon A\to\r^3$ by the expression
\[
X=\Re\big(\int\Phi \big)
\]
($\Re(\cdot)$ denotes real part); hence, a conformal non-orientable minimal immersion $\sub{X}\colon \sub{A}\to\r^3.$ (See \cite{Meeks-8pi}.) 
By definition, the couple $(\Phi,\I)$ is said to be the {\em Weierstrass representation} of $\sub{X}.$

\begin{remark} \label{re:periodos}
A vectorial holomorphic $1$-form $\Phi$ on (an open neighborhood of) $A$ satisfying  \eqref{eq:Wdata-non} has no real periods if and only if
$$\text{ $\int_\gamma \Phi=0\quad$ for any $\gamma \in \Hcal_1(A,\z)$ with $\I_*(\gamma)=\gamma$.}$$
\end{remark}

To finish this subsection we present the {\em flux} of a conformal minimal immersion.

\begin{definition} \label{def:conor}
Let $A$ be an $\I$-invariant subset in $\Ncal,$ let $X\in\Mcal_\I(A),$ (see Def.\ \ref{def:M(A)}) and let $\gamma(s)$ be an arc-length parameterized curve in $A.$ The {\em conormal vector field}  of $X$ along $\gamma$ is the unique unitary tangent vector field  $\mu$ of $X$ along $\gamma$ such that $\{d X(\gamma'(s)),\mu(s)\}$ is a positive basis for all $s.$

If $\gamma$ is  closed, the number $\pgot_X(\gamma):=\int_\gamma \mu(s) ds$ is said to be  the {\em flux}  of $X$ along $\gamma$. 
\end{definition}

Given an $\I$-invariant subset $A$ in $\Ncal,$ and $X\in\Mcal_\I(A),$ it is easy to check that $\pgot_X(\gamma)=\Im\int_{\gamma} \partial X$
(here $\Im(\cdot)$ denotes imaginary part), and that the flux map $$\pgot_X\colon \Hcal_1(A,\z)\to \r^3$$ is a group morphism. Furthermore, since $X$ is $\I$-invariant and $\I$ reverses the orientation, then the flux map $\pgot_X\colon \Hcal_1(A,\z)\to\r^3$ of $X$ satisfies 
\begin{equation}\label{eq:flux-non}
\pgot_X(\I_*(\gamma))=-\pgot_X(\gamma)\quad \forall \gamma\in \Hcal_1(A,\z) .
\end{equation}
By definition, the couple $(\pgot_X,\I)$ is said to be the flux map of $\sub{X}.$


\section{Admissible subsets for the Mergelyan approximation}\label{sec:admi}

We begin this section by describing the subsets involved in  the Mergelyan approximation  theorem in Sec.\ \ref{sec:Mer}. Although there is room for generalizations, the  sets considered in Def.\ \ref{def:admi}  are sufficient for our geometric applications.

\begin{remark}\label{rem:inicio}
From now on in the paper, $\sub{\Ncal},$ $\Ncal,$ $\pi,$ and $\I,$ will be as in Def.\ \ref{def:non}. Moreover, $\sigma_\Ncal^2$ will denote a conformal Riemannian metric on $\Ncal$ such that $\I^*(\sigma_\Ncal^2)=\sigma_\Ncal^2.$
\end{remark}

First of all, recall that a subset $A\subset\Ncal$ is said to be {\em Runge} (in $\Ncal$)  if $\Ncal\setminus A$ has no relatively compact connected components. 

A compact Jordan arc in $\Ncal$ is said to be analytical (smooth, continuous, etc.) if it is contained in an open analytical (smooth, continuous, etc.) Jordan arc in $\Ncal.$

\begin{definition}\label{def:admi}
A (possibly non-connected) $\I$-invariant compact subset $S\subset \Ncal$ is said to be $\I$-admissible in $\Ncal$ if and only if (see Fig.\ \ref{fig:admi}):
\begin{enumerate}[\rm (a)]
\item $S$ is Runge,
\item $R_S:=\overline{S^\circ}$ is non-empty and consists of a finite collection of pairwise disjoint compact regions in $\Ncal$ with   $\Ccal^0$ boundary,
\item $C_S:=\overline{S\setminus R_S}$ consists of a finite collection of pairwise disjoint analytical Jordan arcs, and
\item any component $\alpha$ of $C_S$  with an endpoint  $P\in R_S$ admits an analytical extension $\beta$ in $\Ncal$ such that the unique component of $\beta\setminus\alpha$ with endpoint $P$ lies in $R_S.$
\end{enumerate}
\end{definition}
\begin{figure}[ht]
    \begin{center}
    \scalebox{0.25}{\includegraphics{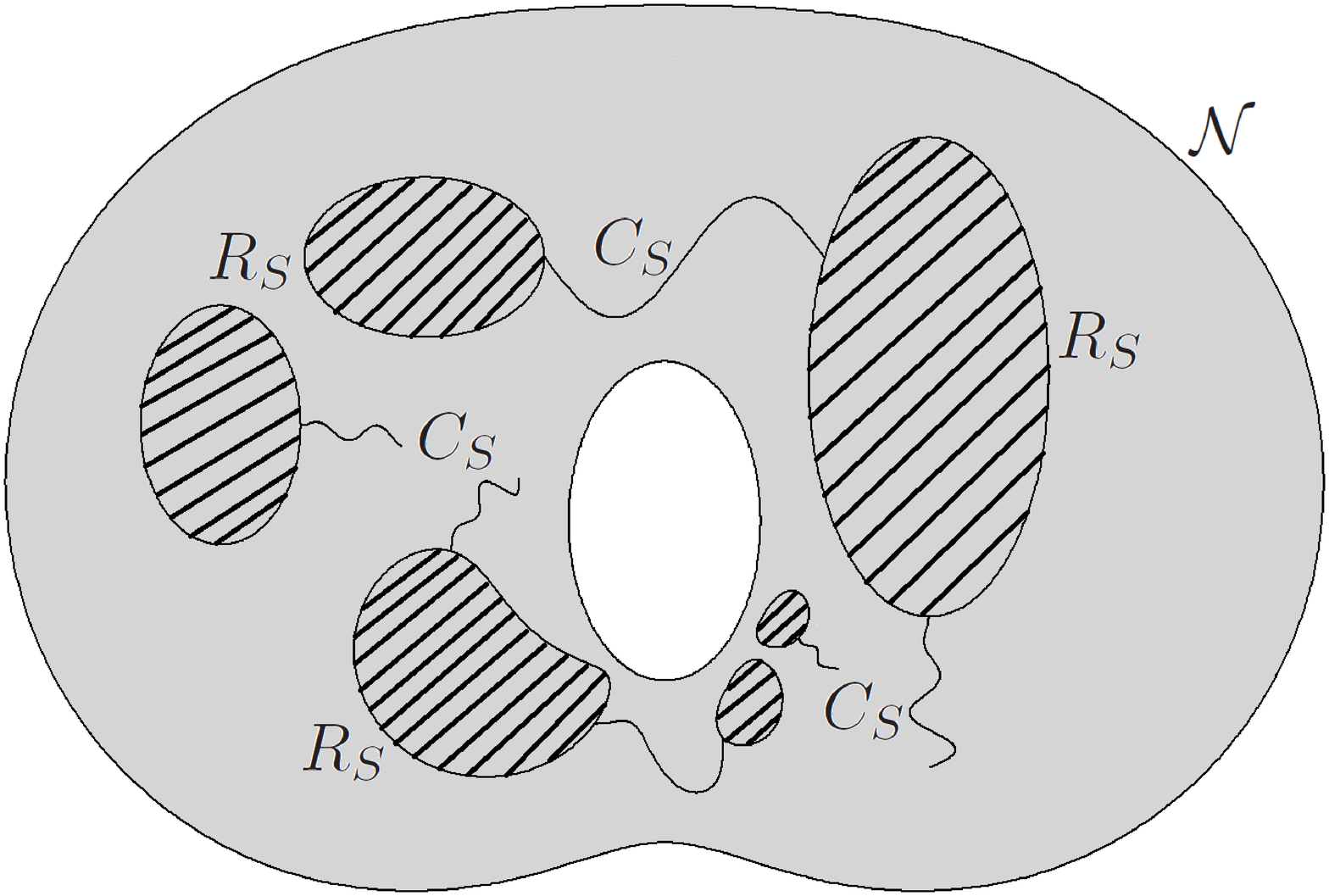}}
        \end{center}
\caption{An $\I$-admissible set $S\subset\Ncal.$}\label{fig:admi}
\end{figure}

An $\I$-invariant compact subset $S\subset \Ncal$ satisfying (b), (c), and (d),  is Runge (hence $\I$-admissible) if and only if $i_*\colon \Hcal_1(S,\z) \to \Hcal_1(\Ncal,\z)$ is a monomorphism, where $\Hcal_1(\cdot,\z)$ means first homology group, $i\colon S \to \Ncal$  is the inclusion map,  and $i_*$ is the induced group morphism. If $S\subset \Ncal$ is an $\I$-invariant compact Runge subset consisting of a finite collection of pairwise disjoint compact regions with $\Ccal^0$ boundary, then $S$  is $\I$-admissible; that is to say, we allow $C_S$ to be empty. The most typical $\I$-admissible subsets $S$ in $\Ncal$ consist of a finite collection of pairwise disjoint compact regions $R_S$ with $\Ccal^1$ boundary, and a finite collection of Jordan analytical arcs $C_S$ meeting $b R_S$ transversally.


\subsection{Functions on $\I$-admissible subsets}

From now on in this section, $S$ will denote an $\I$-admissible subset in $\Ncal,$ in the sense of Def.\ \ref{def:admi}.  

\begin{definition} We denote by
\begin{itemize}
\item $\Fcal_{\hgot,\I}(S)$ the real vectorial space of continuous functions $f\colon S \to \c$, holomorphic on an open neighborhood of $R_S$ in $\Ncal,$ such that $f\circ\I|_S=\overline{f},$ and

\item $\Fcal_{\mgot,\I}(S)$ the real vectorial space of continuous functions $f\colon S \to \overline{\c}$, meromorphic on an open neighborhood of $R_S$ in $\Ncal,$ satisfying that $f\circ\I|_S=\overline{f}$ and  $f^{-1}(\infty)\subset S^\circ=R_S\setminus b R_S.$
\end{itemize}

Likewise, we denote by
\begin{itemize}
\item $\Gcal_\I(S)$ the family of continuous functions $g\colon S \to \overline{\c}$, meromorphic on an open neighborhood of $R_S$ in $\Ncal,$ satisfying that $g\circ\I|_S=-1/\overline{g}$ and  $g^{-1}(\{0,\infty\})\subset S^\circ=R_S\setminus b R_S.$
\end{itemize}
\end{definition}

A $1$-form $\theta$ on $S$ is said to be of type $(1,0)$ if for any conformal chart $(U,z)$ in $\Ncal,$ $\theta|_{U \cap S}=h(z) dz$ for some function $h\colon U \cap S \to \overline{\c}.$
Finite sequences $\Theta=(\theta_1,\ldots,\theta_n),$ where $\theta_j$ is a $(1,0)$-type $1$-form for all $j\in\{1,\ldots,n\},$ are said to be $n$-dimensional vectorial $(1,0)$-forms on $S.$   The space of continuous $n$-dimensional $(1,0)$-forms on $S$ will be endowed with the $\Ccal^0$ topology induced by the norm  
\begin{equation} \label{eq:norma0}
\|\Theta\|_{0,S}:=\|\frac{\Theta}{\sigma_\Ncal}\|_{0,S}=\max_{S} \big\{ \big(\sum_{j=1}^n |\frac{\theta_j}{\sigma_\Ncal}|^2\big)^{1/2}\big\};
\end{equation}
see Remark \ref{rem:inicio}.

\begin{definition} \label{def:divs}
For any $f \in \Fcal_{\mgot,\I}(S)$ we write  $(f)_0$, $(f)_\infty$, and  $(f)$ for the zero divisor  $(f|_{R_S})_0$, the polar divisor $(f|_{R_S})_\infty$, and the divisor $(f|_{R_S})$, respectively; see \cite{FarkasKra}.
\end{definition}
Notice that all these divisors lie in $\div_\I(R_S)$. Obviously, ${\rm supp}((f)_\infty)=f^{-1}(\infty)\subset R_S$ and ${\rm supp}((f)_0)=f^{-1}(0)\cap R_S.$ Likewise we define the corresponding  divisors for functions $g\in\Gcal_\I(S)$, but in this case they do not lie in $\div_\I(R_S)$ unless $(g)=1$. In fact, $\I((g)_0)=(g)_\infty$.

The following Gunning-Narasimhan's type result for relatively compact  $\I$-invariant domains is required for later purposes. A general theorem in this line for non-orientable Riemann surfaces will be shown later in Sec.\ \ref{sec:apli}; see Theorem \ref{th:GN}.

\begin{proposition}\label{pro:Icero}
Let $W$ be a relatively compact  $\I$-invariant open subset in $\Ncal.$ Then there exists a nowhere-vanishing holomorphic $1$-form $\tau$ on $W$ such that $\I^*(\tau)=\overline{\tau}.$
\end{proposition}
\begin{proof}
Since the same argument applies
separately to each connected component, we may assume that $W$ is a domain.

Take a nowhere-vanishing holomorphic $1$-form $\tau_0$ on $\Ncal$ (see \cite{GunningNarasimhan}). If $\tau_0+\overline{\I^*(\tau_0)}$ vanishes everywhere on $W$, then  it suffices to set  $\tau:=\imath \tau_0|_W.$ Otherwise, by the Identity Principle $\tau_1:=(\tau_0+\overline{\I^*(\tau_0)})|_W$ has finitely many zeros on the compact set $\overline{W};$ hence on $W.$ Denote by $D$ the divisor associated to $\tau_1|_{\overline{W}}$. Since $\tau_1\in \Omega_{\hgot,\I}(W)$ we can write $D=D_1 \I(D_1)$, where ${\rm supp}(D_1) \cap {\rm supp}(\I(D_1))=\emptyset.$ Since $\Ncal\setminus\overline{W}$ is a non-empty open set, then the  Riemann-Roch Theorem furnishes a meromorphic function $h$ on $\Ncal$ such that $h|_{\overline{W}}$ is holomorphic and $(h|_{\overline{W}})_0=D_1.$ Set $H:=h \cdot \overline{h\circ \I}\in \Fcal_{\hgot,\I}(\Ncal)$ and observe that $(H|_{\overline{W}})_0=D.$ We finish by setting $\tau:=\tau_1/(H|_W).$
\end{proof}

From now on in this section, let $W$  and $\tau$ be as in Proposition \ref{pro:Icero} such that $S\subset W$. The following notions do not depend on the chosen $W$ and $\tau.$
 
\begin{definition} We denote by
\begin{itemize}
\item $\Omega_{\hgot,\I}(S)$ the real vectorial space of $1$-forms $\theta$ of type $(1,0)$ on $S$ such that $\theta/\tau \in \Fcal_{\hgot,\I}(S),$ and 
\item $\Omega_{\mgot,\I}(S)$ the real vectorial space of $1$-forms $\theta$ of type $(1,0)$ on $S$ such that $\theta/\tau\in \Fcal_{\mgot,\I}(S).$  
\end{itemize}
\end{definition}

Define as above the associated divisors $(\theta)_0$ and $(\theta)_\infty$ of zeros and poles, respectively,  for any $\theta \in \Omega_{\mgot,\I}(S)$. Likewise, denote by $(\theta)=\frac{(\theta)_0}{(\theta)_\infty}$ the divisor of $\theta$, and notice that all these divisors lie in $\div_{\I}(S).$

\begin{definition}\label{def:omega-t} 
We shall say that
\begin{itemize}
\item  a function $f \in \Fcal_{\hgot,\I}(S)$ can be approximated in the $\Ccal^0$ topology on $S$ by functions in $\Fcal_{\hgot,\I}(W)$ if
there exists $\{f_n\}_{n \in \n} \subset \Fcal_{\hgot,\I}(W)$ such that $\{\|f_n|_S-f\|_{0,S}\}_{n \in \n} \to 0,$

\item a function $f \in \Fcal_{\mgot,\I}(S)$ can be approximated in the $\Ccal^0$ topology on $S$ by functions in $\Fcal_{\mgot,\I}(W)$ if
there exists $\{f_n\}_{n \in \n} \subset  \Fcal_{\mgot,\I}(W)$ such that $f_n|_S-f \in {\mathcal{ F}_\hgot}(S)$ for all $n$ and $\{\|f_n|_S-f\|_{0,S}\}_{n \in \n} \to 0$ (in particular, $(f_n)_\infty=(f)_\infty$ on $S^\circ$ for all $n$),

\item  a $1$-form $\theta \in \Omega_{\hgot,\I}(S)$ can be approximated in the $\Ccal^0$ topology on $S$ by $1$-forms in $\Omega_{\hgot,\I}(W)$ if
there exists  $\{\theta_n\}_{n \in \n} \subset \Omega_{\hgot,\I}(W)$ such that $\{\|\theta_n|_S-\theta\|_{0,S}\}_{n \in \n} \to 0,$

\item a $1$-form $\theta \in \Omega_{\mgot,\I}(S)$ can be approximated in the $\Ccal^0$ topology on $S$ by $1$-forms in $\Omega_{\mgot,\I}(W)$ if
there exists  $\{\theta_n\}_{n \in \n} \subset \Omega_{\mgot,\I}(W)$ such that $\theta_n|_S-\theta \in \Omega_{\hgot,\I}(S)$ for all $n$ and $\{\|\theta_n|_S-\theta\|_{0,S}\}_{n \in \n} \to 0$  (in particular $(\theta_n)_\infty=(\theta)_\infty$  on $S^\circ$ for all $n$), and

\item  a function $g\in\Gcal_\I(S)$ can be approximated in the $\Ccal^0$ topology on $S$ by functions in $\Gcal_\I(W)$ if there exists $\{g_n\}_{n \in \n} \subset \Gcal_\I(W)$ such that $g_n-g$ is holomorphic on  (a neighborhood of) $R_S$ and  $\{\|g_n|_S-g\|_{0,S}\}_{n \in \n} \to 0.$    
\end{itemize}
\end{definition}

We define the notions of approximation in the $\Ccal^0$ topology of vectorial functions and $1$-forms in a similar way. 

\begin{definition} \label{def:f-smooth}
A function $f\colon S\to\k^n$ ($\k=\r,$ $\c,$ or $\overline{\c},$ $n\in\n$) is said to be {\em smooth}
if  $f|_{R_S}$ admits a smooth extension $f_0$ to an open domain $V$ in $\Ncal$ containing $R_S,$ and for any component $\alpha$ of $C_S$ and any open analytical Jordan arc $\beta$ in $\Ncal$ containing $\alpha,$  $f|_\alpha$ admits a smooth extension $f_\beta$ to $\beta$ satisfying that $f_\beta|_{V \cap \beta}=f_0|_{V \cap \beta}.$
\end{definition}

\begin{definition} \label{def:t-smooth}
A vectorial $1$-form $\Theta$ of type $(1,0)$ on $S$  is said to be {\em smooth} if $\Theta/\tau\colon S \to \overline{\c}^n$ is smooth in the sense of Def.\ \ref{def:f-smooth}.
\end{definition}

\begin{definition}
Given a smooth function $f\in\Fcal_{\mgot,\I}(S)\cup \Gcal_\I(S),$ we denote by $df$ the $1$-form of type $(1,0)$  given by 
\[
df|_{R_S}=\partial (f|_{R_S})\enskip \text{and}\enskip df|_{\alpha \cap U}=(f \circ \alpha)'(x)dz|_{\alpha \cap U}
\]
for any component $\alpha$ of $C_S,$ where $(U,z=x+\imath y)$ is any conformal chart on $\Ncal$ satisfying that $z(\alpha \cap U)\subset \r\equiv \{y=0\}$ (the existence of such a conformal chart is guaranteed by the analyticity of $\alpha$).
Notice that $df$ is well defined and smooth. Furthermore, $df|_\alpha(t)= (f\circ\alpha)'(t) dt$ for any component $\alpha$ of $C_S,$ where $t$ is any smooth parameter along $\alpha.$ 
\end{definition}

If  $f\in\Fcal_{\mgot,\I}(S)$ is a smooth function, then $df$ belongs to $\Omega_{\mgot,\I}(S)$ (to $\Omega_{\hgot,\I}(S)$ if $f\in\Fcal_{\hgot,\I}(S)$). 


A smooth $1$-form $\theta \in \Omega_{\mgot,\I}(S)$  is said to be {\em exact} if $\theta=df$ for some smooth $f \in \Fcal_{\mgot,\I}(S),$ or equivalently if $\int_\gamma \theta=0$ for all $\gamma \in \Hcal_1(S,\z).$ The exactness of vectorial $1$-forms in $\Omega_{\mgot,\I}(S)^n$, $n\in\n$, is defined in the same way.


\subsection{Conformal minimal immersions on $\I$-admissible subsets}\label{sec:minimal-I}

Let us begin this subsection by generalizing the notion of conformal minimal immersion to maps defined on $\I$-admissible sets; see Def.\ \ref{def:admi}, and also Def.\ \ref{def:saturated} for notation.

\begin{definition}
A map $\sub{X}\colon \sub{S} \to \r^3$ is said to be a {\em generalized non-orientable minimal immersion} if the $\I$-invariant map 
\[
X:=\sub{X}\circ\pi|_S\colon S\to\r^3\quad \text{(see Def.\ \ref{def:invariant})}
\]
is smooth (see Def.\ \ref{def:f-smooth}), and satisfies that 
\begin{itemize}
\item $X|_{R_S} \in \Mcal_\I(R_S)$ (see Def.\ \ref{def:M(A)}) and
\item $X|_{C_S}$ is regular; that is to say, $X|_\alpha$ is a regular curve for all $\alpha \subset C_S.$
\end{itemize}
In this case, we also say that $X$ is an {\em $\I$-invariant generalized minimal immersion}, and write $X\in\Mcal_{\ggot,\I}(S).$
\end{definition}

Notice that $X|_{S}\in \Mcal_{\ggot,\I}(S)$  for all $X \in \Mcal_\I(S).$ 

Let $X\in \Mcal_{\ggot,\I}(S)$  and let $\varpi\colon C_S \to \r^3$ be a {\em smooth normal field} along $C_S$ with respect to $X$; this means that for any (analytical) arc-length parametrized $\alpha (s) \subset C_S,$   $\varpi(\alpha(s))$ is smooth, unitary, and orthogonal to $(X\circ \alpha)'(s),$    $\varpi$ extends smoothly to any open analytical arc $\beta$ in $\Ncal$ containing $\alpha$, and $\varpi$ is tangent to $X$ on $\beta \cap S$.  

Let ${\mathfrak{n}}\colon R_S\to\s^2$ denote the Gauss map of the (oriented) conformal minimal immersion $X|_{R_S}$.  The normal field $\varpi$ is said to be {\em orientable} with respect to $X$ if  for any  regular embedded curve $\alpha \subset S$ and arc-length parametrization $X\circ \alpha(s)$ of $X\circ \alpha$, there exists a constant $\delta\in\{-1,1\}$ (depending on the parametrization) such that  
$$\text{$(X\circ\alpha)'({s_0})\times  \varpi(\alpha({s_0}))=\delta {\mathfrak{n}}(\alpha({s_0}))\quad$ for all ${s_0}\in \alpha^{-1}(C_S\cap R_S)$.}$$ If $\varpi$ is orientable, $\alpha$ is a connected component of $C_S$, and $\delta=1$,  then $s$ is said to be a {\em positive} arc-length parameter of $X\circ \alpha$ with respect to $\varpi$. Positive arc-length parameters  with respect to $\varpi$ on regular curves in $C_S$ are unique up to translations.

If $\varpi$ is orientable with respect to $X$, we denote by ${\mathfrak{n}}_\varpi\colon S\to \s^2\subset \r^3$ the smooth map given by ${\mathfrak{n}}_\varpi|_{R_S}={\mathfrak{n}}$ and $({\mathfrak{n}}_\varpi \circ \alpha)(s):=(X\circ \alpha)'(s)\times \varpi(\alpha(s))$, where $\alpha$ is any component of
$C_S$ and $s$ is the positive arc-length parameter of $X\circ\alpha$  with respect to $\varpi$.  By definition,  ${\mathfrak{n}}_\varpi$ is said to be the {\em (generalized) Gauss map} of $X$ associated to the orientable smooth normal field $\varpi$. 
Obviously, if $\varpi$ is orientable then $-\varpi$ is orientable as well and   ${\mathfrak{n}}_\varpi= {\mathfrak{n}}_{-\varpi}$.

\begin{definition}\label{def:marked}
We denote by  $\Mcal_{\ggot,\I}^*(S)$ the space of marked immersions $X_\varpi:=(X,\varpi),$ where $X \in \Mcal_{\ggot,\I}(S)$ and $\varpi$ is an  orientable smooth normal field along $C_S$ with respect to $X$ such that 
$\varpi\circ \I=-\varpi$, or equivalently,
\begin{equation}\label{eq:con-non}
{\mathfrak{n}}_\varpi \circ \I=-{\mathfrak{n}}_\varpi.
\end{equation}
\end{definition}

\begin{remark}
Marked minimal immersions play the role of $\I$-invariant conformal minimal immersions of  $\I$-admissible subsets into $\r^3.$ They will be the natural initial conditions for the Mergelyan approximation theorem  in Sec.\ \ref{sec:Mer}.
\end{remark}

Let $X_\varpi \in \Mcal_{\ggot,\I}^*(S),$ and let $\partial
X_\varpi=(\hat{\phi}_j)_{j=1,2,3}$ be the complex vectorial
$1$-form  on $S$ given by  
\[
\partial X_\varpi|_{R_S}=\partial (X|_{R_S}),\quad \partial X_\varpi(\alpha'(s))= dX
(\alpha'(s)) + \imath \varpi(s);
\]
where $\alpha$ is a component of
$C_S$ and $s$ is the positive arc-length parameter of $X\circ\alpha$ with respect to $\varpi$. If $(U,z=x+\imath y)$ is a
conformal chart on $\Ncal$ such that $\alpha \cap U=z^{-1}(\r \cap
z(U)),$ it is clear that $(\partial X_\varpi)|_{\alpha \cap
U}=\big[dX (\alpha'(s)) + \imath \varpi(s)\big]s'(x)dz|_{\alpha \cap
U},$ hence
 $\partial X_\varpi \in \Omega_{\hgot,\I}(S)^3.$
 
Furthermore, the function
\[
\hat{g}\colon S \to \overline{\c},\quad \hat{g}=\frac{\hat{\phi}_3}{\hat{\phi}_1-\imath\hat{\phi}_2},
\]
is continuous on $S,$ meromorphic on an open neighborhood of $R_S$ in $\Ncal,$ and formally satisfies \eqref{eq:g-non}; hence $\hat{g} \in \Gcal_\I(S)$ provided that $\hat{g}^{-1}(\{0,\infty\})\subset R_S.$ Further,  $\hat{g}$ is nothing but the stereographic projection of the Gauss map ${\mathfrak{n}}_\varpi$ of $X_\varpi$.

Obviously, $\hat{\phi}_j$ is smooth on $S,$ $j=1,2,3,$ and the
same occurs for $\hat{g}$ provided that $\hat{g}^{-1}(\{0,\infty\})\subset R_S$
 (see Def.\ \ref{def:f-smooth} and \ref{def:t-smooth}). In addition, $\partial X_\varpi$ formally satisfies \eqref{eq:conformal}, \eqref{eq:Wdata-non}, \eqref{eq:immersion}, and  $\Re (\hat{\phi}_j)$ is an exact real $1$-form on $S,$
$j=1,2,3;$ hence we also have that $X(P)=X(Q)+\Re \int_{Q}^P
(\hat{\phi}_j)_{j=1,2,3},$ $P,$ $Q \in S.$
 For these reasons, $(\hat{g},\hat{\phi}_3)$ will be said the {\em generalized Weierstrass data} of $X_\varpi.$ Since $\partial X_\varpi$ and  $\hat{g}$  formally satisfy \eqref{eq:Wdata-non} and \eqref{eq:g-non}, then one can introduce the {\em generalized complex Gauss map} $\hat{G}\colon \sub{S}\to\r\p^2$ of $\sub{X}$ associated to $\varpi;$ see Def.\ \ref{def:gauss-non}. 

Notice that $X|_{R_S} \in \Mcal_\I (R_S),$ hence $(\phi_j)_{j=1,2,3}:=(\hat{\phi}_j|_{R_S})_{j=1,2,3}$ and $g:=\hat{g}|_{R_S}$ are obviously the  Weierstrass data and the complex Gauss map of $X|_{R_S},$ respectively.

The space $\Mcal_{\ggot,\I}^*(S)$ is naturally endowed with the following $\Ccal^1$ topology:

\begin{definition}\label{def:C1}
Given $X_{\varpi},$ $Y_{\xi}\in\Mcal_{\ggot,\I}^*(S),$ we set
\[
\|X_{\varpi}-Y_{\xi}\|_{1,S}:=\|X-Y\|_{0,S}+\big\|\partial X_{\varpi}-\partial Y_{\xi}\big\|_{0,S} \quad(\text{see \eqref{eq:norma0}}).
\]

Given $F \in \Mcal_\I(S)$, we denote by $\varpi_F$ the conormal field of $F$ along  $C_S.$ Notice that $\varpi_F$ satisfies \eqref{eq:con-non} and $\partial F|_S=\partial F_{\varpi_F},$ where $F_{\varpi_F}:=(F|_S,\varpi_F) \in \Mcal_{\ggot,\I}^*(S).$ 

Given $F,$ $G\in \Mcal_\I(S),$ we set 
\[
\|F-X_{\varpi}\|_{1,S}:=\|F_{\varpi_F}-X_{\varpi}\|_{1,S}\quad \text{and} \quad \|F-G\|_{1,S}:=\|F_{\varpi_F}-G_{\varpi_G}\|_{1,S}.
\]
\end{definition}

\begin{definition} \label{def:C1-top}
We shall say that $X_\varpi \in \Mcal_{\ggot,\I}^*(S)$  can be  approximated in the $\Ccal^1$ topology on $S$ by $\I$-invariant conformal minimal immersions in $\Mcal_\I(W)$ if for any $\epsilon>0$ there exists  $Y\in  \Mcal_\I(W)$ such that $\|Y-X_\varpi\|_{1,S}<\epsilon.$
\end{definition} 

If $X_\varpi\in\Mcal_{\ggot,\I}^*(S),$ then the group homomorphism  
\[
\pgot_{X_\varpi}\colon \Hcal_1(S,\z) \to \r^3, \quad \pgot_{X_\varpi}(\gamma)=\Im \int_\gamma \partial X_\varpi,
\]
is said to be the {\em generalized flux map} of $X_\varpi.$ Notice that $\pgot_{X_\varpi}$  satisfies \eqref{eq:flux-non}. Obviously,  $\pgot_{X_{\varpi_Y}}=\pgot_Y|_{\Hcal_1(S,\z)}$ provided that  $X=Y|_S$ for some $Y\in\Mcal_\I(S).$


\section{Approximation results}\label{sec:approx}

Throughout this section, $S\subset \Ncal$ will denote an $\I$-admissible subset, and  $W$ a relatively compact  $\I$-invariant open subset of $\Ncal$ containing $S.$ 

In this section we state and prove several preliminary approximation results that will be key in the proof of the main theorem, in Sec.\ \ref{sec:Mer}. In particular, Lemmas \ref{lem:funaprox} and \ref{lem:gaprox} deal with functions in $\Fcal_{\mgot,\I}(S)$ and $\Gcal_\I(S)$, respectively. We begin with the following

\begin{lemma} \label{lem:runge} For any  $f \in \Fcal_{\mgot,\I}(S)$ and integral divisor $D \in \div_\I(S)$ with  ${\rm supp}(D)\subset S^\circ,$ there exists $\{f_n\}_{n \in \n} \subset \Fcal_{\mgot,\I}(W)$ such that $f_n|_S-f \in \Fcal_{\hgot,\I}(S)$ and $\big(f_n|_S-f\big) \geq D$ for all $n\in\n,$ and $\{\|f_n|_S-f\|_{0,S}\}_{n \in \n} \to 0.$
\end{lemma}
\begin{proof}
By classical approximation results (see \cite[Theorem 4.1]{AL-proper} for details), there exists a sequence of meromorphic functions $\{h_n\colon W\to\overline\c\}_{n\in\n}$ such that $h_n|_S-f\colon S\to\c$ is continuous on $S$ and holomorphic in a neighborhood of $R_S,$ $(h_n|_S-f)\geq D,$ $\forall n\in\n$ (see Def.\ \ref{def:divs}), and $\{\|h_n|_S-f\|_{0,S}\}_{n\in\n}\to 0.$ Since $f\circ \I=\bar f$ and $D$ is $\I$-invariant, then the sequence $\{\overline{h_n\circ\I}\colon W\to\c\}_{n\in\n}$ meets the same properties. Therefore, it suffices to set $f_n:=\frac12(h_n+\overline{h_n\circ\I})$ for all $n\in\n.$
\end{proof}

Let us prove the following deeper approximation result for functions in  $\Fcal_{\mgot,\I}(S)$:

\begin{lemma} \label{lem:funaprox}
Let $f \in \Fcal_{\mgot,\I}(S)$ vanishing nowhere on $S\setminus S^\circ=(b R_S) \cup C_S,$ and let $D\in \div_\I(S)$ an integral divisor with  ${\rm supp}(D)\subset S^\circ.$  

Then there exists    $\{f_n\}_{n \in \n} \subset \Fcal_{\mgot,\I}(W)$ satisfying that $f_n|_S-f \in \Fcal_{\hgot,\I}(S)$, $(f_n)=(f)$ on $W$, and $\big(f_n|_S-f\big) \geq D$ for all $n\in\n,$ and $\{\|f_n|_S-f\|_{0,S}\}_{n \in \n} \to 0.$  In particular, $f_n$ is  holomorphic and vanishes nowhere on $W\setminus S^\circ$ for all $n\in\n.$
\end{lemma}

\begin{proof}
Since the space of smooth functions is dense in $\Fcal_{\mgot,\I}(S)$ under the $\Ccal^0$ topology, we can assume that $f$ is smooth. Furthermore, since the same argument applies
separately to each connected component, we may assume that $W$ is a domain.

Let us begin the proof with the following reduction:
\begin{claim}
It suffices to prove the lemma just for  nowhere-vanishing functions in $\Fcal_{\hgot,\I}(S)$.
\end{claim}
\begin{proof}
Assume that the lemma holds for nowhere-vanishing functions in $\Fcal_{\hgot,\I}(S).$ 

Take any $f\in \Fcal_{\mgot,\I}(S)$ and $D \in \div_\I(S)$ as in the statement of the lemma,  and write $(f)=D_1 \I(D_1),$ where  ${\rm supp}(D_1) \cap {\rm supp}(\I(D_1))=\emptyset.$ Since $W$ is relatively compact in the open Riemann surface $\Ncal,$ then the Riemann-Roch theorem furnishes  a meromorphic function $h_0$ on $\Ncal$ with $(h_0|_{\overline{W}})=D_1.$  The function $F:=f/(h_0( \overline{h_0\circ \I}))$ lies in $\Fcal_{\hgot,\I}(S)$ and vanishes nowhere on $S.$ By our assumption, there exists a sequence  of nowhere vanishing functions $\{F_n\}_{n\in \n}$ in $\Fcal_{\hgot,\I}(W)$ approximating $F$ on $S$ and satisfying $(F-F_n)\geq D_2$, where $D_2\in \div_\I(S)$ is any given integral  divisor with 
${\rm supp}(D_2)\subset S^\circ$ and $D_2\geq D D_1^{-1}\I(D_1)^{-1}$. 

If we choose $f_n:=F_n h_0 (\overline{h_0\circ \I})\in \Fcal_{\mgot,\I}(W)$, one has that $f_n|_S-f\in \Fcal_{\hgot,\I}(S)$, $(f_n|_S-f)\geq D$ and $(f_n)=(f)$ for all $n\in\n,$ and $\{f_n\}_{n\in \n}$ approximates $f$ on $S$.
\end{proof}

In the sequel we will assume that $f\in \Fcal_{\hgot,\I}(S)$ and has no zeros.

By the isomorphism \eqref{eq:isom}, there exists $\tau\in \Omega_{\hgot,\I}(W)$ such that 
\begin{equation}\label{eq:tau-exact}
\int_\gamma df/f=\int_\gamma \tau \enskip  \forall \gamma \in \Hcal_1(S,\z) ,
\end{equation}
and 
\begin{equation} \label{eq:res0}
\frac{1}{2 \pi \imath} \int_\gamma \tau \in \z \enskip  \forall \gamma \in \Hcal_1(W,\z).
\end{equation}
Here we have taken into account that $f$ is assumed to be smooth, $\Hcal_1(S,\z)$ is a natural subgroup of $\Hcal_1(W,\z)$ (recall that $S$ is Runge in $W$)  and $\frac{1}{2\pi \imath} \int_\gamma df/f \in \z$ for all $\gamma \in \Hcal_1(S,\z)$.

We need the following
\begin{claim}\label{cla:funciones}
There exist a nowhere vanishing  function $v\in \Fcal_{\hgot,\I}(W)$ and a function $u\in \Fcal_{\hgot,\I}(S)$ such that 
$d \log(v)=\tau$, $du=f/df-\tau|_S$,  and$f=v e^u$.
\end{claim}
\begin{proof}
To construct $v$, fix $P_0\in W$ and notice that  $\Re \int_{P_0}^{\I(P_0)}\tau=0$ independently on the arc connecting $P_0$ and $\I(P_0).$ Indeed, take any oriented Jordan arc $\gamma\subset W$  with initial point $P_0$ and final point $\I(P_0),$ and simply observe that
\[
\overline{\int_\gamma \tau}=\int_\gamma \overline{\tau}=\int_\gamma \I^*(\tau)=\int_{\I_*(\gamma)} \tau=-\int_\gamma \tau+2 k \pi \imath, \quad k \in \z .
\]
For the last equality, take into account  that $\gamma+\I_*(\gamma)\in \Hcal_1(W,\z)$ and \eqref{eq:res0}. 

Therefore, the well defined function $v:=e^{\int_{{P_0}} \tau-\frac12 \int_{{P_0}}^{{\I(P_0)}}\tau}$  lies in $\Fcal_{\hgot,\I}(W)$, and obviously satisfies $d \log(v)=\tau$.

To construct $u$, recall that $df/f-\tau|_S$ lies in $\Omega_{\hgot,\I}(S)$ and is exact; see \eqref{eq:tau-exact}. For each  connected component $C$ of $S$, fix $P_C\in C$ and set  $u|_{C}:=A_C+ \int_{P_{C}} (f/df-\tau)$, where the constant $A_C\in \c$ is chosen so that $(f-ve^{u|_C})(P_C)=0$. Since the function $\kappa:=f/(e^u v)$ is locally constant on $S$ (just observe that $d \log(\kappa)=0$) and $\kappa(P_C)=1$ for any connected component $C$ of $S$, we infer that $\kappa=1$, that is to say, $f=v e^u $ on $S$.

The facts $f,\,v|_S\in \Fcal_{\hgot,\I}(S)$ imply that $u\circ \I=\overline{u} +2 m\pi \imath$ for some $m\in \z.$ Up to replacing $u$ and $v$ for $u-m \pi \imath$ and $e^{m \pi \imath} v$, respectively,  we get that  $u\in \Fcal_{\hgot,\I}(S)$ and the functions $u$ and $v$ solve the claim.
\end{proof}

Let $u\in \Fcal_{\hgot,\I}(S)$ and $v\in \Fcal_{\hgot,\I}(W)$ like in the previous claim. By Lemma \ref{lem:runge}, there exists $\{h_n\}_{n \in \n}
\subset \Fcal_{\hgot,\I}(W)$ such that $(h_n-u)\geq D$ for all $n$, $\{\| h_n|_S- u\|_{0,S}\}_{n \in \n} \to 0.$  
To conclude, it suffices to set  $f_n:=e^{h_n}v$ for all $n.$
\end{proof}

We now derive the analogous approximation result for $1$-forms in  $\Omega_{\mgot,\I}(S)$.

\begin{lemma} \label{lem:formaprox}
Let $\theta \in \Omega_{\mgot,\I}(S)$  vanishing nowhere on $S\setminus S^\circ$,  and consider an integral divisor $D \in \div_\I(S)$ with  ${\rm supp}(D)\subset S^\circ$.

Then  there exists $\{\theta_n\}_{n \in \n} \in \Omega_{\mgot,\I}(W)$ satisfying that $\theta_n-\theta\in \Omega_{\hgot,\I}(S)$, $(\theta_n-\theta)\geq D$, and $(\theta_n)=(\theta)$ on $W$ for all $n\in\n$, and $\{\|\theta_n|_S-\theta\|_{0,S}\}_{n \in \n} \to 0.$   In particular, $\theta_n$ is  holomorphic and vanishes nowhere on $W\setminus S^\circ$ for all $n\in\n.$
\end{lemma}
\begin{proof} 
Let $\tau\in\Omega_{\hgot,\I}(W)$ having no zeros (see Proposition \ref{pro:Icero}). Label $f=\theta/\tau \in \Fcal_{\mgot,\I}(S),$ and notice that $(f)=(\theta);$ in particular  $f$ has no zeros on $S\setminus S^\circ.$ By Lemma \ref{lem:funaprox}, there exists $\{f_n\}_{n \in \n}$ in $\Fcal_{\mgot,\I}(W)$ such that  $\{\|f_n|_S-f\|_{0,S}\}_{n \in \n} \to 0$ and  $(f_n)=(f)$ and $(f_n-f)\geq D$
on $W$ for all $n\in\n.$ It suffices to set $\theta_n:=f_n \tau\in \Omega_{\mgot,\I}(W)$ for all $n\in \n.$
\end{proof}

We finish this section by proving a similar approximation result for functions in $\Gcal_\I(S)$.

\begin{lemma} \label{lem:gaprox}
Let $g \in \Gcal_\I(S)$ vanishing nowhere on $S\setminus S^\circ$,  and let $D\in \div_\I(S)$ an integral divisor with  ${\rm supp}(D)\subset S^\circ.$  

Then there exists  $\{g_n\}_{n \in \n} \subset \Gcal_\I(W)$ satisfying that $g_n-g$ is holomorphic on $R_S,$ $(g_n-g)\geq D$,
$(g_n)=(g)$ on $W$ for all $n\in\n,$ and $\{\|g_n|_S-g\|_{0,S}\}_{n \in \n} \to 0.$  In particular, $g_n$ is  holomorphic and vanishes nowhere on $W\setminus S^\circ$ for all $n\in\n.$
\end{lemma}

\begin{proof}
Since  smooth functions are dense in $\Gcal_{\mgot,\I}(S)$   with respect to the $\Ccal^0$ topology, we can suppose without loss of generality that $g$ is smooth.  Furthermore, since the same argument applies
separately to each connected component, we may assume that $W$ is a domain.

\begin{claim}
It suffices to prove the lemma just for  nowhere-vanishing functions in $\Gcal_{\I}(S)$.
\end{claim}
\begin{proof} Assume that $(g)\neq 1$. Then, consider a non constant meromorphic function $h$ on $\Ncal$ satisfying $h\circ \I=1/\overline{h}$ and $(h|_{\overline{W}})=(g)$. To construct $h$, 
write $ (g)=D_1 \I(D_1)^{-1},$ where $D_1$ is an integral divisor and ${\rm supp}(D_1) \cap {\rm supp}(\I(D_1))=\emptyset.$ Since $\Ncal\setminus \overline{W}$ is open, the Riemann-Roch theorem provides a meromorphic function $H$ on $\Ncal$ with $(H|_{\overline{W}})=D_1.$ Setting $h= {H}/{\overline{H \circ \I}}$ we are done.

The function  $f=g/h\colon S\to\overline{\c}$ lies in $\Gcal_\I(S)$ and is nowhere vanishing. By our assumption, there exists a sequence of nowhere vanishing functions $\{f_n\}_{n\in \n}$ in $\Gcal_{\I}(W)$ approximating $f$ on $S$, and satisfying that $f_n|_S-f$ is holomorphic on $R_S$ and $(f_n-f)\geq D (g)^{-1}$ for all $n$.

Choosing $g_n:=f_n h\in \Gcal_{\mgot,\I}(W)$, one has that $g_n|_S-g$ is holomorphic on $R_S$, $(g_n|_S-g)\geq D$ and $(g_n)=(g)$ for all $n\in\n,$ and $\{g_n\}_{n\in \n}$ approximates $g$ on $S$.
\end{proof}

%

%

In the sequel we will suppose that $g$ is nowhere vanishing.

Notice that $\imath dg/g\in \Omega_{\hgot,\I}(W)$.  Reasoning as in the proof of Lemma \ref{lem:funaprox},  there exists a holomorphic $1$-form $\tau\in \Omega_{\hgot,\I}(W)$ such that
\begin{equation}\label{eq:tau-exact1}
\int_\gamma \imath dg/g=\int_\gamma \tau \enskip  \forall \gamma \in \Hcal_1(S,\z)
\end{equation}
and 
\begin{equation} \label{eq:res0bis}
\frac{1}{2 \pi} \int_\gamma \tau \in \z \enskip  \forall \gamma \in \Hcal_1(W,\z).
\end{equation}

We need the following
\begin{claim}\label{cla:funciones1}
There exist a nowhere vanishing holomorphic $v\in \Gcal_\I(W)$ and a function $u\in \Fcal_{\hgot,\I}(S)$ such that 
$d \log(v)=-\imath \tau$, $du=\imath dg/g-\tau|_S$,  and $g=v e^{-\imath u}$.
\end{claim}
\begin{proof}
To construct  $v$, fix $P_0\in W$  and  reasoning as in Claim \ref{cla:funciones} observe that  $ \frac1{\pi}\Re \big(  \int_{{P_0}}^{{\I(P_0)}}\tau \big) \in \z$ independently on the arc connecting $P_0$ and $\I(P_0)$. The function 
$v:=e^{-\imath \int_{{P_0}} \tau+\frac{\imath}2 \int_{{P_0}}^{{\I(P_0)}}  \tau}$ is well defined (see \eqref{eq:res0bis}), nowhere vanishing,  and satisfies $d \log(v)=-\imath \tau$. Furthermore, 
\begin{equation} \label{eq:dicotomia}
\text{$v (\overline{v \circ \I})=\pm 1,$ depending on whether $\frac{1}{\pi} \big( \Re \int_{P_0}^{\I(P_0)}\tau \big)$ is even or odd.}
\end{equation}


To construct $u$ we proceed as in Claim \ref{cla:funciones}. Take into account that $\imath dg/g-\tau|_S$ lies in $\Omega_{\hgot,\I}(S)$ and is exact, see \eqref{eq:tau-exact1}. For each  connected component $C$ of $S$, fix $P_C\in C$ and set  $u|_{C}:=A_C+ \int_{P_{C}} (\imath dg/g-\tau)$, where $A_C\in \c$ is chosen so that $(g-v e^{-\imath u})(P_C)=0$. Since  $\kappa:=g/(v e^{-\imath u})$ is locally constant on $S$ and $\kappa(P_C)=1$ for any connected component $C$ of $S$, we infer that $\kappa=1$ and $g=v e^{-\imath u}$ on $S$.  

The facts $g\in \Gcal_{\I}(S)$ and \eqref{eq:dicotomia} imply that $u\circ \I=\bar u +m\pi$ for some $m\in \z.$ Since $\I$ is an involution, we infer that $m=0$,  $u\in \Fcal_{\hgot,\I}(S)$, and $v (\overline{v \circ \I})=-1$; see  \eqref{eq:dicotomia}. 

The functions $u$ and $v$ solve the claim.
\end{proof}

By Lemma \ref{lem:runge}, there exists a sequence $\{h_n\}_{n \in \n}
\subset \Fcal_{\hgot,\I}(W)$ such that  $\{\| h_n|_S-u\|_{0,S}\}_{n \in \n} \to 0$ and $(h_n-u)_0\geq D$ for all $n\in \n$.
The sequence of nowhere vanishing functions $\{g_n:= e^{-\imath h_n}v\}_{n\in \n}\subset \Gcal_{\hgot,\I}(W)$ proves the lemma. 
\end{proof}


\section{Runge-Mergelyan's Theorem for non-orientable minimal surfaces}\label{sec:Mer}

In this section we prove the main result of the paper (Theorem \ref{th:Mergelyan}). Most of the technical computations are contained in the following Lemma \ref{lem:weiaprox}; Theorem \ref{th:Mergelyan} will follow by a recursive application of it. In particular, Lemma \ref{lem:weiaprox} deals with the problem of controlling the periods in the approximation procedure.

\begin{lemma} \label{lem:weiaprox}
Let $S\subset \Ncal$ be an $\I$-admissible subset, let $W$ be a relatively compact $\I$-invariant domain in $\Ncal$ with finite topology, containing $S$, and let $\Phi=(\phi_j)_{j=1,2,3}$ be a smooth triple in $\Omega_{\hgot,\I}(S)^3$ such that $\sum_{j=1}^3 \phi_j^2=0$ and $\sum_{j=1}^3 |\phi_j|^2$ vanishes nowhere on $S.$  

Then  $\Phi$ can be approximated in the $\Ccal^0$ topology on $S$ by a sequence $\{\Phi_n=(\phi_{j,n})_{j=1,2,3}\}_{n \in \n}\subset \Omega_{\hgot,\I}(W)^3$ meeting the following requirements:
\begin{enumerate}[\rm (i)]
  \item $\phi_{3,n}$ vanishes nowhere on $W\setminus R_S$. Furthermore, $(\phi_{3,n}|_E)=(\phi_3|_E)\in\div(E)$ for any connected component $E$ of $R_S$ such that $\phi_3$ does not vanish everywhere on $E$.
  \item $\sum_{j=1}^3 \phi_{j,n}^2=0$ and $\sum_{j=1}^3 |\phi_{j,n}|^2$ vanishes nowhere on $W.$
  \item $\Phi_n-\Phi$ is exact on $S$ for all $n\in\n.$
\end{enumerate}
\end{lemma}

\begin{proof} 
Denote by 
\[
g=\frac{\phi_3}{\phi_1- \imath \phi_2} ,
\]
and recall that $g\circ \I= -1/\overline{g}$; see Sec.\ \ref{sec:minimal-I}. This implies in particular that $g|_{R_S}$ is not constant, but it could be locally constant. We rule out this possibility in the following claim.

\begin{claim} \label{ass:gauss}
Without loss of generality, it can be assumed that $g|_E$ is not constant for any connected component $E$ of $R_S$.
\end{claim}
\begin{proof} 
Assume that the lemma holds when $g$ is non-constant on every connected component of $R_S$, and let us show that it also holds in the general case.

We discuss first the following particular case.

\noindent{\em Case 1.} If $g$ vanishes everywhere on no connected component of $R_S$, then the lemma holds.

Indeed, let $E$ be a connected component of $R_S$ such that $g|_E$ is constant. By our assumption,  
\begin{equation} \label{eq:nowhere0}
\text{$g|_E\neq 0,\infty$, and $\phi_3|_E$ is nowhere vanishing.}
\end{equation}
Since $g\circ \I= -1/\overline{g}$, then $E\cap \I(E)=\emptyset$. Label $E_1,\I(E_1),\ldots, E_k,\I(E_k)$ the family of connected components $E$ of $R_S$ such that $g|_E$ is constant.

Denote by $\Lambda_1 =\cup_{j=1}^k E_j$ and by $\Lambda_2 =R_S\setminus (\Lambda_1\cup\I(\Lambda_1))$. Let $\Bcal_1$ be a homology basis of $\Hcal_1(\Lambda_1,\z)$ and denote by $\nu_1\in \n$ the number of elements in $\Bcal_1.$ Denote by $\Ocal(\Lambda_1)$ the space of holomorphic functions $\Lambda_1\to\c.$ 

For each $h\in \Ocal(\Lambda_1),$ consider the holomorphic function and $1$-form on $\Lambda_1$ given by
\begin{equation}\label{eq:g(h)}
g(h):=(g+h) \quad \text{and}\quad  \phi_3(h):=\frac{g+h}{g} \phi_3 .
\end{equation}

Let $\Pcal\colon \Ocal(\Lambda_1) \to \c^{2 \nu_1}$ be the {\em period map} given by 
\[
\Pcal(h)=\Big(\int_c \big( g(h)\phi_3(h)-g\phi_3 \,,\, \phi_3(h)-\phi_3 \big) \Big)_{c \in \Bcal_1}.
\]
Notice that $\Pcal(\lambda h)=0$ for any $\lambda \in \c$ and any $h\in \Ocal(\Lambda_1)$ with $\Pcal(h)=0$. Since $\Ocal(\Lambda_1)$ has infinite dimension, then there exists a non-constant function $h \in \Pcal^{-1}(0)\subset\Ocal(\Lambda_1).$ Let $\{\lambda_n\}_{n \in \n}\subset \c\setminus \{0\}$ be any sequence converging to zero, and define 
\begin{equation}\label{eq:T=0}
h_n:=\lambda_n h\in \Pcal^{-1}(0)\quad \text{for all $n\in\n$}.
\end{equation}
Obviously, $\{h_n\}_{n \in \n}\to 0$ in the $\Ccal^0$ topology on $\Lambda_1$; hence without loss of generality we can assume that 
\begin{equation} \label{eq:nowhere}
\text{$g(h_n)$ and $\phi_3(h_n)$ vanish nowhere on $\Lambda_1$ for all $n$.}
\end{equation}
Choose smooth $g_n\in\Gcal_\I(S)$ and $\phi_{3,n}\in \Omega_{\hgot,\I}(S)$ such that 
\[
g_n|_{\Lambda_1}=g(h_n),\enskip g_n|_{\Lambda_2}=g,\enskip \phi_{3,n}|_{\Lambda_1}=\phi_3(h_n),\enskip \phi_{3,n}|_{\Lambda_2}=\phi_3 ,
\]
and $\{\Psi_n\}_{n\in\n}\subset\Omega_{\hgot,\I}(S)$ converges to $\Phi$ in the $\Ccal^0$ topology on $S$; where $\Psi_n$ are the Weierstrass data associated to $(g_n,\phi_{3,n})$ via \eqref{eq:recoverX}. Notice that $(\phi_{3,n})=(\phi_3)$ for all $n$; see \eqref{eq:nowhere0} and \eqref{eq:nowhere}. Observe that \eqref{eq:g(h)} and \eqref{eq:T=0} imply that $\Psi_n-\Phi$ is exact on $R_S=\Lambda_1\cup\I(\Lambda_1)\cup\Lambda_2$. Furthermore, up to a slight smooth deformation of $g_n$ and $\phi_{3,n}$ over $C_S$, we can also assume that $\Psi_n-\Phi$ is exact on $S$ for all $n\in\n$. 

By our assumptions, the lemma holds for the triple $\Psi_n$, for all $n\in\n$. To finish, use a standard diagonal argument.

\noindent{\em Case 2.} If $g$ vanishes everywhere on some connected components of $R_S$, then the lemma holds. 

Call  $\Lambda_0\neq \emptyset$ the union of those connected components of $R_S$ on which $g$ is identically $0$ or $\infty$, and notice that $\phi_3|_{\Lambda_0}$ vanishes everywhere and $\Phi|_{\Lambda_0}$ is exact. Take a sequence $\{A_n\}_{n \in \n} \subset \Ocal(3,\r)$ converging to the identity matrix such that the third coordinate of $(\Phi|_{\Lambda_0})\cdot A_n$ vanishes everywhere on no connected component of $\Lambda_0$.  Choose a smooth $\Theta_n \in \Omega_{\hgot,\I}(S)^3$ such that  $\Theta_n|_{\Lambda_0}:=(\Phi|_{\Lambda_0})\cdot A_n$, $\Theta_n|_{R_S\setminus \Lambda_0}:=\Phi|_{R_S\setminus \Lambda_0}$, $\Theta_n-\Phi$ is exact on $S$, and $\{\Theta_n\}_{n\in \n}\to \Phi$ in the $\Ccal^0$ topology on $S$. Since the third coordinate of $\Theta_n$ vanishes everywhere on no connected component of $R_S$, the conclusion of Lemma 
\ref{lem:weiaprox} holds for each $\Theta_n$; take into account Case 1. By a diagonal argument, it also holds for $\Phi$ and we are done.

This proves the claim.
\end{proof}

From now on, we assume that $g$ is non-constant on every connected component of $R_S$. Let us check the following

\begin{claim} \label{ass:motion}
Without loss of generality,  it can be assumed that $\phi_j$ and $d(\phi_i/\phi_j)$ vanish nowhere on $(b R_S) \cup C_S$, for all $i,j\in\{1,2,3\}$, $i\neq j$. In particular, $g$ vanishes nowhere on $C_S$; hence $g  \in \Gcal_\I(S)$.
\end{claim}
\begin{proof}
Let $M_1 \supset M_2 \supset \ldots$ be a sequence of $\I$-invariant compact regions in $W$ such that $M_n^\circ$ is a tubular neighborhood of $R_S$ in $W$ for all $n\in\n$,  $M_n \subset M_{n-1}^\circ$ for all $n\in\n$, $\cap_{n \in  \n} M_n=R_S,$ $\Phi$ holomorphically extends to $M_1$,  $\sum_{j=1}^3 |\phi_j|^2$ vanishes nowhere on $M_1$, and  $\phi_j$ and $d(\phi_i/\phi_j)$ vanish nowhere on $b M_n$, for all $i,j\in\{1,2,3\}$, $i\neq j$, $n\in\n$  (recall that $g$ is  constant over no connected component of $R_S$ by Claim \ref{ass:gauss}). In addition, choose $M_n$ so that $S_n:=M_n \cup C_S\subset W$ is an $\I$-admissible set in $\Ncal$ and $\gamma\setminus M_n^\circ$ is a non-empty Jordan arc, for any component $\gamma$ of $C_S.$ Observe that $C_{S_n}= C_S\setminus M_n^\circ$, for all $n \in \n.$

Let  $(h_n,\psi_{3,n})\in \Gcal_\I(S_n) \times \Omega_{\hgot,\I}(S_n)$ be any smooth couple meeting the following requirements:
\begin{itemize}
\item  $(h_n,\psi_{3,n})|_{R_{S_n}}=(g,\phi_3)|_{R_{S_n}}$ and $\sum_{j=1}^3 |\psi_{j,n}|^2$ vanishes nowhere on $S_n$; where $\Psi_n=(\psi_{j,n})_{j=1,2,3}\in \Omega_{\hgot,\I}(S_n)^3$ are the Weierstrass data associated to $(h_n,\psi_{3,n})$ via \eqref{eq:recoverX}.

\item  $\psi_{j,n}$ and $d(\psi_{i,n}/\psi_{j,n})$ vanish nowhere on $(b R_{S_n}) \cup C_{S_n}$, for all $i,j\in\{1,2,3\}$, $i\neq j$.

\item $\Psi_n|_S-\Phi$ is exact on $S$.

\item the sequence $\{\Psi_n|_S\}_{n \in \n} \subset \Omega_{\hgot,\I}(S)^3$ converges to $\Phi$ in the $\Ccal^0$ topology on $S.$
\end{itemize}
The existence of such sequence follows from similar arguments as those in Claim \ref{ass:gauss}; just use classical approximation results by smooth functions to suitably extend the couple $(g,\phi_3)|_{R_{S_n}}$ to $C_{S_n}$.

If we assume that the lemma holds for the triple $\Psi_n$ and the $\I$-admissible set $S_n$, for all $n\in\n$, then, using again a standard diagonal argument, we conclude that it also holds for the triple $\Phi$.
\end{proof}
From now on, we assume that $\phi_j$ and $d(\phi_i/\phi_j)$ vanish nowhere on $(b R_S) \cup C_S$, for all $i,j\in\{1,2,3\}$, $i\neq j$.

Let $W_S\subset W$ be a domain of finite topology  such that $S\subset W_S$ and $i_*\colon \Hcal_1(S,\z)\to \Hcal_1(W_S,\z)$ is an isomorphism, where as usual $i\colon S\to W_S$ denotes the inclusion map.  
Denote by $\nu=2 \nu_0+1,$ $ \nu_0\in\n$, the number of generators of $\Hcal_1(S,\z)$ of $S$. Take an $\I$-basis $\Bcal_S=\{c_0,c_1,\ldots,c_{\nu_0},d_1,\ldots,d_{\nu_0}\}$ of the homology group with real coefficients $\Hcal_1(S,\r)\stackrel{i_*}{\equiv} \Hcal_1(W_S,\r)$; see  Def.\ \ref{def:ibasis}. 

Recall that 
\begin{equation}\label{eq:c-d}
\I_*(c_j)=-c_j\quad \text{and}\quad \I_*(d_j)=d_j,\quad \text{for all $j$}.
\end{equation}

For any couple of functions $(h_1,h_2)\in \Fcal_{\hgot,\I}(\overline{W})^2$, denote by $\Phi(h_1,h_2)$ the Weierstrass data on $S$ associated to the pair $(e^{\imath h_1} g,e^{h_2}\phi_3)$ by \eqref{eq:recoverX}; that is,
\[
\Phi(h_1,h_2)=\left( \frac12 \Big(\frac1{e^{\imath h_1}g}-e^{\imath h_1}g \Big) \,,\, \frac{\imath}2 \Big(\frac1{e^{\imath h_1}g}+e^{\imath h_1}g \Big)  \,,\, 1\right)e^{h_2}\phi_3 .
\]

Observe that $\Phi(h_1,h_2)\in \Omega_{\hgot,\I}(S)^3$, and so, by \eqref{eq:c-d},
\begin{equation} \label{eq:simper}
\int_{c_j}\Phi(h_1,h_2)\in\imath\r^3 \quad \text{and}\quad \int_{d_j}\Phi(h_1,h_2)\in\r^3 \quad \text{$\forall j$ and $(h_1,h_2)$.}
\end{equation}
The same happens in particular to the triple $\Phi=\Phi(0,0)$.

Denote by $\Pcal\colon \Fcal_{\hgot,\I}(\overline{W})^2 \to \r^{6\nu_0+3}\equiv \r^{3\nu_0+3}\times \r^{3\nu_0}$ the {\em period map} given by
\begin{equation}\label{eq:periodmap}
\Pcal (h_1,h_2)=
\Big( \Big[\Im\int_{c_j} \Phi(h_1,h_2)-\Phi\Big]_{j\in\{0,\ldots,\nu_0\}}  \,,\, \Big[\int_{d_j} \Phi(h_1,h_2)-\Phi \Big]_{j\in\{1,\ldots,\nu_0\}}\Big).
\end{equation}
Notice that $\Phi(h_1,h_2)$ satisfies items {\rm (i)} and {\rm (ii)} in the lemma; if in addition $\Pcal(h_1,h_2)=0$, then it also meets item {\rm (iii)} (take into account \eqref{eq:simper}). On the other hand, endowing the real space $\Fcal_{\hgot,\I}(\overline{W})^2$ with the maximum norm, one has that the period map $\Pcal$ above is Fr\'echet differentiable. 

The key to the proof of Lemma \ref{lem:weiaprox} is to show that the Fr\'echet derivative of $\Pcal$ has maximal rank equal to $6\nu_0+3$ at $(0,0)$.

\begin{claim} \label{ass:regular}
The Fr\'echet derivative $\Acal\colon \Fcal_{\hgot,\I}(\overline{W})^2\to \r^{6 \nu_0+3}$ of $\Pcal$ at $(0,0)$ has maximal rank.
\end{claim}
\begin{proof} 
We proceed by contradiction. Assume that $\Acal(\Fcal_{\hgot,\I}(\overline{W})^2)$ is contained in a linear subspace
\begin{multline*}
\Ucal=\Big\{\big([(x_{c_j,k})_{k=1,2,3}]_{j=0,\ldots,\nu_0},[(x_{d_j,k})_{k=1,2,3}]_{j=1,\ldots,\nu_0}\big) \in \r^{6 \nu_0+3} \colon\\
\sum_{k=1}^3\Big(\sum_{j=0}^{\nu_0}   \lambda_{c_j,k} x_{c_j,k}+ \sum_{j=1}^{\nu_0}  \lambda_{d_j,k} x_{d_j,k}\Big)=0\Big\}\subset\r^{6\nu_0+3} ,
\end{multline*}
where 
$\lambda_{c_j,k}$ and $\lambda_{d_j,k}$ are real numbers for all $j$ and $k$, not all of them equal to zero. Let $\Gamma_k$ be the element of the homology group with complex coefficients $\Hcal_1(S,\c)$ given by
\begin{equation}\label{eq:Gamma}
\Gamma_k=-\imath\sum_{j=0}^{\nu_0}\lambda_{c_j,k}c_j +\sum_{j=1}^{\nu_0}\lambda_{d_j,k}d_j,\quad k=1,2,3\,.
\end{equation}

Since $\Acal(\Fcal_{\hgot,\I}(\overline{W})^2)\subset\Ucal$, then
\begin{equation} \label{eq:fun0}
-\int_{\Gamma_1} h\phi_2+\int_{\Gamma_2}h\phi_1=0\quad \text{for all $h\in \Fcal_{\hgot,\I}(\overline{W})$}
\end{equation}
and
\begin{equation} \label{eq:fun}
\int_{\Gamma_1} h\phi_1+\int_{\Gamma_2} h\phi_2 +\int_{\Gamma_3} h\phi_3=0\quad \text{for all $h\in \Fcal_{\hgot,\I}(\overline{W})$.}
\end{equation}

Let us show that $\Gamma_1=0.$

Indeed, reason by contradiction and assume that $\Gamma_1\neq 0$. Denote by $\Sigma_1=\{f \in \Fcal_{\hgot,\I}(\overline{W}) \colon (f) \geq (\phi_1)^2\}.$ By Claim \ref{ass:motion}, the function $df/\phi_1$ lies in $\Fcal_{\hgot,\I}(S).$ Therefore, for any $f \in \Sigma_1,$ Lemma \ref{lem:runge} applies and insures that $df/\phi_1$ can be approximated in the $\Ccal^0$ topology on $\Fcal_{\hgot,\I}(S)$ by functions in  $\Fcal_{\hgot,\I}(\overline{W})$. As a consequence, equation \eqref{eq:fun0} can be applied formally to $h=df/\phi_1$, implying that 
\[
\int_{\Gamma_1} \frac{\phi_2}{\phi_1} df=0\quad\text{for all $f \in \Sigma_1.$}
\]
By Claim \ref{ass:motion} one can integrate by parts in the above equation and obtain that
\begin{equation} \label{eq:fun1}
\int_{\Gamma_1} f d\Big(\frac{\phi_2}{\phi_1}\Big) =0 \quad\text{for all $f \in \Sigma_1.$}
\end{equation} 

Since $\Gamma_1\neq 0$, the isomorphism \eqref{eq:isom}  gives a holomorphic $1$-form $\tau\in\Omega_{\hgot,\I}(\overline{W})$ such that 
\begin{equation}\label{eq:contra1}
\int_{\Gamma_1}\tau \in \r\setminus\{0\};
\end{equation}
take into account \eqref{eq:Gamma} and \eqref{eq:c-d}.

On the other hand, since $W$ is open and relatively compact in $\Ncal$, then there exists $u \in  \Fcal_{\hgot,\I}(\overline{W})$ such that
$(\tau+du)_0 \geq (\phi_1)^2 (d(\phi_2/\phi_1))$; use Riemann-Roch theorem. Set 
\[
f_0:=\frac{\tau+du}{d(\phi_2/\phi_1)}\in\Fcal_{\hgot,\I}(S)
\]
(see Claim \ref{ass:motion}) and note that $(f_0) \geq (\phi_1)^2$. By Lemma \ref{lem:runge}, $f_0$ can be approximated in the $\Ccal^0$ topology on $\Fcal_{\hgot,\I}(S)$ by functions in  $\Sigma_1$; hence equation \eqref{eq:fun1} can be applied formally to $f=f_0,$ giving that $0=\int_{\Gamma_1} (\tau+du)=\int_{\Gamma_1} \tau$; contradicting \eqref{eq:contra1}.

Therefore,  $\Gamma_1=0$ and equation \eqref{eq:fun0} becomes
\begin{equation}\label{eq:fun00}
\int_{\Gamma_2}h\phi_1=0\quad \text{for all $h\in\Fcal_{\hgot,\I}(\overline{W})$}.
\end{equation}

Next we show that $\Gamma_2=0$. As above, reasoning by contradiction, we can find a $1$-form $\tau\in\Omega_{\hgot,\I}(\overline{W})$ and a function $u \in  \Fcal_{\hgot,\I}(\overline{W})$ such that 
\begin{equation}\label{eq:contra2}
\int_{\Gamma_2} \tau \neq 0
\end{equation}
and $(\tau+du)_0 \geq  (\phi_2).$ In this case, we set 
\[
h=\frac{\tau+du}{\phi_2}\in \Fcal_{\hgot,\I}(S);
\]
see Claim \ref{ass:motion}. By Lemma \ref{lem:runge}, one can approximate $h$ in the $\Ccal^0$ topology on $\Fcal_{\hgot,\I}(S)$ by functions in $\Fcal_{\hgot,\I}(\overline{W})$; hence \eqref{eq:fun00} formally applies to $h$ giving that $0=\int_{\Gamma_2} (\tau+du)=\int_{\Gamma_2} \tau$; which contradicts \eqref{eq:contra2}.

Finally, since $\Gamma_1=\Gamma_2=0$, then \eqref{eq:fun} becomes $\int_{\Gamma_3} h\phi_3=0$ for all $h\in \Fcal_{\hgot,\I}(\overline{W})$. The same argument as that in the previous paragraph gives that $\Gamma_3=0$ as well.

Since $\Gamma_k=0$ for all $k=1,2,3$, then \eqref{eq:Gamma} implies that $\lambda_{c_j,k}=0=\lambda_{d_j,k}$ for all $j$ and $k$. This contradiction finishes the proof.
\end{proof}

Let $\{e_1,\ldots,e_{6 \nu_0+3}\}$ be a basis of $\r^{6 \nu_0+3}$. For any $j\in\{1,\ldots,6\nu_0+3\}$ choose
$H_j=(h_{1,j},h_{2,j}) \in \Acal^{-1}(e_j)\subset\Fcal_{\hgot,\I}(\overline{W})^2$, and denote by $\Qcal\colon \r^{6 \nu_0+3} \to \r^{6 \nu_0+3}$ the analytical map given by
\[
\Qcal((x_j)_{j=1,\ldots,6 \nu_0+3})=\Pcal\Big(\sum_{j=1}^{6\nu_0+3} x_j H_j\Big) ,
\] 
where $\Pcal$ is the period map \eqref{eq:periodmap}. Claim \ref{ass:regular} guarantees that the differential of $\Qcal$ at $0\in\r^{6\nu_0+3}$ is an isomorphism; hence there exists a closed Euclidean ball $U\subset
\r^{6 \nu_0+3}$ centered at the origin, satisfying that $\Qcal\colon U \to
\Qcal(U)$ is an analytical diffeomorphism. In particular,
$0=\Qcal(0)$ lies in the interior of  $\Qcal(U).$

On the other hand, taking into account Claim \ref{ass:motion}, Lemmas \ref{lem:formaprox} and \ref{lem:gaprox} furnish a sequence $\{(\sigma_n,\psi_{3,n})\}_{n
\in \n} \subset \Gcal_{\I}(\overline{W})\times \Omega_{\hgot,\I}(\overline{W})$ such that
\begin{equation}\label{eq:regular}
(\sigma_n)=(g)\quad \text{and}\quad (\psi_{3,n})=(\phi_3)\in\div_\I(R_S)\quad \text{for all $n\in\n$},
\end{equation}
and
\begin{equation}\label{eq:conv1}
\{(\sigma_n,\psi_{3,n})|_S\}_{n \in \n} \to (g,\phi_3)\quad \text{in the
$\Ccal^0$ topology on $S$}.
\end{equation}

For any couple of functions $(h_1,h_2)\in \Fcal_{\hgot,\I}(\overline{W})^2$, we denote by $\Psi_n(h_1,h_2)\in\Omega_{\hgot,\I}(\overline{W})^3$ the Weierstrass data associated to the pair $(e^{\imath h_1} \sigma_n,e^{h_2}\psi_{3,n})$ by \eqref{eq:recoverX}; that is to say,
\[
\Psi_n(h_1,h_2)=\left( \frac12 \Big(\frac1{e^{\imath h_1}\sigma_n}-e^{\imath h_1}\sigma_n \Big) \,,\, \frac{\imath}2 \Big(\frac1{e^{\imath h_1}\sigma_n}+e^{\imath h_1}\sigma_n \Big)  \,,\, 1\right)e^{h_2}\psi_{3,n} .
\]

By \eqref{eq:c-d}, one has that $\int_{c_j}\Psi_n(h_1,h_2)\in\imath\r^3$ and $\int_{d_j}\Psi_n(h_1,h_2)\in\r^3$, for all $j$ and $(h_1,h_2)\in\Fcal_{\hgot,\I}(\overline{W})^2$. Denote by $\Pcal_n\colon \Fcal_{\hgot,\I}(\overline{W})^2 \to \r^{6 \nu_0+3}\equiv \r^{3\nu_0+3}\times\r^{3\nu_0}$ the period map given by
\begin{equation}\label{eq:Pn}
\Pcal_n (h_1,h_2)=
\Big( \Big[\Im\int_{c_j} \Psi_n(h_1,h_2)-\Phi\Big]_{j\in\{0,\ldots,\nu_0\}}  \,,\, \Big[\int_{d_j} \Psi_n(h_1,h_2)-\Phi \Big]_{j\in\{1,\ldots,\nu_0\}}\Big),
\end{equation}
and notice that $\Pcal_n$ is Fr\'{e}chet differentiable if we endow the real space $\Fcal_{\hgot,\I}(\overline{W})^2$ with the maximum norm.

Denote by $\Qcal_n\colon \r^{6 \nu_0+3} \to \r^{6 \nu_0+3}$ the analytical map given by
\[
\Qcal_n((x_j)_{j=1,\ldots,6 \nu_0+3})=\Pcal_n\Big(\sum_{j=1}^{6\nu_0+3} x_j H_j\Big)\quad \text{for all $n \in \n$}.
\]
Since $\{\Qcal_n\}_{n \in \n} \to \Qcal$  uniformly on compact subsets of $\r^{6\nu_0+3}$, then $\Qcal_n\colon U \to \Qcal_n(U)$ is an analytical diffeomorphism  and  $0 \in \Qcal_n(U)$ for all $n\geq n_0$ for some $n_0\in\n$. Denote by $y_n=(y_{j,n})_{j=1,\ldots,6 \nu_0+3}$ the unique point in $U$ mapped to $0$ by $\Qcal_n$, $n\geq n_0$. Since $\Qcal(0)=0$, then 
\begin{equation}\label{eq:conv2}
\text{the sequence $\{y_n\}_{n \geq n_0}$ converges to $0$}.
\end{equation}

Set
\[
g_n:=e^{\sum_{j=1}^{6 \nu_0+3} y_{j,n} h_{1,j}} \sigma_n\in \Gcal_{\I}(\overline{W})\quad \text{and}\quad \phi_{3,n}:=e^{\sum_{j=1}^{6 \nu_0+3} y_{j,n} h_{2,j}} \psi_{3,n}\in \Omega_{\hgot,\I}(\overline{W}),
\]
for all $n\geq n_0$. Denote by $\Phi_n$ the Weierstrass data on $W$ associated to the pair $(g_n,\phi_{3,n})$ by \eqref{eq:recoverX}, $n\geq n_0$, and let us check that the sequence $\{\Phi_n\}_{n \geq n_0}$ solves the lemma. Indeed, $\{\Phi_n\}_{n\geq n_0}$ converges to $\Phi$ in the $\Ccal^0$ topology on $S$ by \eqref{eq:conv1} and \eqref{eq:conv2}. Since $\Qcal_n(y_n)=0$, then $\Phi_n-\Phi$ is exact on $S$. Finally, conditions Lemma \ref{lem:weiaprox}-{\rm (i)} and {\rm (ii)} are ensured by \eqref{eq:regular}.
\end{proof}

By a minor modification of the proof of Lemma \ref{lem:weiaprox}, we can obtain the analogous approximation result for Weierstrass data with a fixed component $1$-form. This will be very useful for applications; see Sec. \ref{sec:apli}.

\begin{lemma} \label{lem:weiaprox2} In Lemma \ref{lem:weiaprox} one can choose $\phi_{3,n}=\phi_3$ for all $n \in \n,$ provided that $\phi_3$ extends holomorphically to  $W$, vanishes everywhere on no connected component of $R_S$, and vanishes nowhere on $C_S.$
\end{lemma}
\begin{proof} 
Reasoning as in the proof of Claim \ref{ass:motion}, it can be assumed without loss of generality that $\phi_j$ and $d(\phi_i/\phi_j)$ vanish nowhere on $(b R_S) \cup C_S$, for all $i,j\in\{1,2,3\}$, $i\neq j$. In this case, we take $\I$-admissible sets $S_n$ as those in the proof of Claim \ref{ass:motion}, and replace the Weierstrass data $(g,\phi_3)$ on $S_n$ by suitable $(h_n,\phi_3)$ for all $n\in\n$.

As in Claim \ref{ass:regular}, one can now check that the Fr\'{e}chet derivative $\hat\Acal\colon \Fcal_{\hgot,\I}(\overline{W}) \to \r^{4\nu_0+2}$ of
the period map $\hat\Pcal\colon \Fcal_{\hgot,\I}(\overline{W}) \to \r^{4\nu_0+2}\equiv\r^{4\nu_0+2}\times\{0\}\subset\r^{6\nu_0+3}$, 
$\hat \Pcal(h):=\Pcal(h,0)$, at $h=0$ has maximal rank; where $\Pcal$ is the map \eqref{eq:periodmap}.  Then fix a basis $\{e_1,\ldots, e_{4 \nu_0+2}\}$ of $\r^{4 \nu_0+2},$ and for any $j\in\{1,\ldots,4\nu_0+2\}$ choose a function $\hat H_j \in \hat\Acal^{-1}(e_j)\subset\Fcal_{\hgot,\I}(\overline{W})$ . Denote by $\hat \Qcal\colon \r^{4 \nu_0+2} \to \r^{4 \nu_0+2}$ the analytic map $\hat \Qcal((x_j)_{j=1,\ldots,4 \nu_0+2})=\hat \Pcal(\sum_{j=1,\ldots,4 \nu_0+2} x_j \hat H_j).$

Write
\begin{equation}\label{eq:divphi3}
(\phi_3|_{W\setminus S})=D_1\I(D_1)
\end{equation}
wehre ${\rm supp}(D_1)\cap{\rm supp}(\I(D_1))=\emptyset.$ Since $W$ is relatively compact in $\Ncal$, the Riemann-Roch theorem provides a holomorphic function $H_1\colon \overline{W}\to\c$ such that $(H_1)=D_1.$ Set $H:=H_1/\overline{\I(H_1)}$, and notice that $H$ is a meromorphic function on $\overline{W}$, $(H)=D_1\I(D_1)^{-1}$, and $H\circ\I=1/\overline{H}$. Since $g$ vanishes nowhere on $S\setminus S^\circ$, then $g/H\in\Gcal_{\I}(S)$ does; hence Lemma \ref{lem:gaprox} furnishes a sequence  $\{\rho_n\}_{n \in \n} \subset \Gcal_{\I}(\overline{W})$ such that $(\rho_n)=(g)\in \div(S^\circ)$ for all $n\in\n$ and $\{\rho_n|_S\}_{n \in \n} \to g/H$ in the $\Ccal^0$ topology on $S.$ Set $\sigma_n:=\rho_nH\in \Gcal_{\I}(W)$  and notice that 
\begin{equation}\label{eq:divg}
(\sigma_n)=(g)D_1 \I(D_1)^{-1},\quad \text{for all $n\in\n$},
\end{equation}
and $\{\sigma_n|_S\}_{n \in \n} \to g$ in the $\Ccal^0$ topology on $S.$ 
Observe that \eqref{eq:divphi3} and \eqref{eq:divg} ensure that three $1$-forms of the Weierstrass data associated  by \eqref{eq:recoverX} to the pair $(\sigma_n,\phi_3)$ are holomorphic and have no common zeros.

Denote by $\hat \Pcal_n\colon \Fcal_{\hgot,\I}(\overline{W}) \to \r^{4 \nu_0+2}\equiv\r^{4\nu_0+2}\times\{0\}\subset\r^{6\nu_0+3}$ the period map given by $\hat \Pcal_n(h)=\Pcal_n(h,0)$, where $\Pcal_n$ is the map \eqref{eq:Pn}, and denote by $\hat \Qcal_n\colon \r^{4 \nu_0+2} \to \r^{4 \nu_0+2}$ the analytical map $\hat \Qcal_n((x_j)_{j=1,\ldots,4 \nu_0+2})=\hat \Pcal_n(\sum_{j=1}^{4 \nu_0+2} x_j \hat H_j)$  for all $n \in \n.$  To conclude the proof, we argue as in the  proof of Lemma \ref{lem:weiaprox}.
\end{proof}


We are now ready to state and prove the main result of this paper. Theorem \ref{th:intro} is a particular instance of the following

\begin{theorem}[Runge-Mergelyan's Theorem for non-orientable minimal surfaces]\label{th:Mergelyan}
Let $S\subset \Ncal$ be an $\I$-admissible subset (see Def.\ \ref{def:admi}), let $X_\varpi \in \Mcal_{\ggot,\I}^*(S)$ (see Def.\ \ref{def:marked}), and let $\pgot\colon \Hcal_1(\Ncal,\z)\to\r^3$ be a group homomorphism such that $\pgot(\I_*(\gamma))=-\pgot(\gamma)$ for all $\gamma\in \Hcal_1(\Ncal,\z)$, and  $\pgot|_{\Hcal_1(S,\z)}$ is the generalized flux map $\pgot_{X_\varpi}$ of $X_\varpi$. Write $X_\varpi=(X=(X_j)_{j=1,2,3},\varpi)$, $\partial X_\varpi=(\phi_j)_{j=1,2,3}$, and $\pgot=(\pgot_j)_{j=1,2,3}$.

Then the following assertions hold:

\begin{enumerate}[\rm (I)]
\item $X_\varpi$ can be approximated in the $\Ccal^1$ topology on $S$ by $\I$-invariant conformal minimal immersions $Y=(Y_j)_{j=1,2,3}\colon \Ncal\to\r^3$ such that $\pgot_Y=\pgot$ and $\partial Y_3$ vanishes nowhere on $\Ncal\setminus R_S$. Furthermore,  $Y$ can be chosen so that $(\partial Y_3|_E)=(\phi_3|_E)\in\div(E)$ for any connected component $E$ of $R_S$ such that $\phi_3$ does not vanish everywhere on $E$.
\item If $\phi_3$ is not identically zero and extends to $\Ncal$ as a holomorphic $1$-form without real periods, vanishing nowhere on $C_S$, and satisfying $\pgot_3(\gamma)=\Im\int_\gamma \phi_3$ for all $\gamma\in\Hcal_1(\Ncal,\z)$, then $X_\varpi$ can be approximated in the $\Ccal^1$ topology on $S$ by $\I$-invariant conformal minimal immersions $Y=(Y_j)_{j=1,2,3}\colon \Ncal\to\r^3$ with flux map $\pgot_Y=\pgot$ and third coordinate function $Y_3=X_3$.
\end{enumerate}
\end{theorem}
\begin{proof}
We begin with the following assertion.
\begin{claim}\label{cla:con}
There exists a connected $\I$-admissible subset  $\hat S\subset\Ncal$ such that $R_{\hat S}=R_S$ and $C_{\hat S}\supset C_S$; that is to say, $\hat S$ is constructed by adding a finite family of Jordan arcs to $S$. 
\end{claim}
\begin{proof} If $S$ is connected choose $\hat{S}=S.$ 

Assume that $S$ is not connected. We distinguish the following two cases. (See Remark \ref{rem:inicio} and Def.\ \ref{def:non} and \ref{def:saturated} for notation.)

\noindent {\it Case $1$.} $\sub{S}$ is a connected subset of the non-orientable Riemann surface $\sub{\Ncal}.$ In this situation, any tubular neighborhood of $\sub S$ is an orientable surface. Then,  take any Jordan arc $\sub{\gamma}\subset \sub{\Ncal}$ with end points in $b R_{\sub{S}}$ and otherwise disjoint from $\sub S,$ such that any tubular neighborhood of $\sub{\hat{S}}\colon =\sub S \cup \sub \gamma$ is non-orientable. Since 
$\sub \Ncal \setminus \sub{\hat{S}}$ has no relatively compact connected components, then $\hat{S}$ meets the requirements of the claim.

\noindent {\it Case $2$.} $\sub{S}$ is not connected. Then, consider a finite family $\sub \Gamma$ of pairwise disjoint Jordan arcs in $\sub \Ncal$ such that
\begin{itemize}
\item $\sub{\gamma}$ has end points in $\sub{b R_S}$ and is otherwise disjoint from $\sub S,$ for all $\sub \gamma \in \sub \Gamma,$
\item $\sub S':=\sub S \cup \sub \Gamma$ is connected, and
\item $\sub \Ncal \setminus \sub S' $ has no relatively compact connected components.
\end{itemize}
This reduces the proof of the claim to Case $1$ applied to $S'$.
\end{proof}

Let $\hat S$ be as in Claim \ref{cla:con}.

Let us prove assertion {\rm (I)}. 

Fix $\epsilon>0$. Fix $P_0\in S$ and let $\Phi_0=(\phi_{0,j})_{j=1,2,3}$ be any extension of $\partial X_\varpi$ to $\hat S$ such that 
\begin{enumerate}[\rm (a)]
\item $\phi_{0,j}\in \Omega_{\hgot,\I}(\hat S)$ and is smooth for all $j=1,2,3$,
\item $\Phi_0$ has no real periods, $\sum_{j=1}^3\phi_{0,j}^2=0$, and $\sum_{j=1}^3|\phi_{0,j}|^2$ vanishes nowhere on $\hat S$,
\item $X(P)=X(P_0)+\Re\int_{P_0}^P\Phi_0$ for all $P\in S$, and
\item $\Im\int_\gamma \Phi_0=\pgot(\gamma)$ for all $\gamma\in \Hcal_1(\hat S,\z)$.
\end{enumerate}
To construct $\Phi_0$, just define $\Phi_0$ on the arcs ${C_{\hat S}\setminus C_S}$ in a suitable way. 
Denote by $F_0\in\Mcal_{\ggot,\I}^*(M_0)$ the marked immersion with generalized Weierstrass data $\Phi_0$ and $F_0(P_0)=X_\varpi(P_0)$. 

Set $M_0:=\hat S$ and $M_1$ a tubular neighborhood of $M_0$.  Let $\{M_n\}_{n\in\n}$ be an exhaustion of $\Ncal$ by Runge connected $\I$-invariant compact regions such that the Euler characteristic $\chi(M_n^\circ\setminus M_{n-1})\in\{0,-2\}$ for all $n\geq 2$. Existence of such an exhaustion is well known. Furthermore, $\{M_n\}_{n\in\n}$ meets the following topological description:
\begin{remark}\label{rem:topology}
The region $M_n$ is obtained from $M_{n-1}$, $n\geq 2$, by one of the following four procedures:
\begin{enumerate}[\rm 1.]
\item $M_n$ is a tubular neighborhood of $M_{n-1}$. In this case $\chi(M_n^\circ\setminus M_{n-1})=0$.

\item $M_n$ is a tubular neighborhood of $M_{n-1}\cup \gamma\cup\I(\gamma)$, where $\gamma$ is a Jordan arc in $\Ncal$ with endpoints in a connected component $c$ of $b M_{n-1}$ and otherwise disjoint from $M_{n-1}$, such that $\gamma\cap\I(\gamma)=\emptyset$ and $M_{n-1}\cup \gamma\cup\I(\gamma)$ is an $\I$-admissible subset in $\Ncal$. In this case, $M_n$ has the same genus as $M_{n-1}$ and two more boundary components; hence $\chi(M_n^\circ\setminus M_{n-1})=-2$. (See Fig.\ \ref{fig:topo1}.)
\begin{figure}[ht]
    \begin{center}
    \scalebox{0.50}{\includegraphics{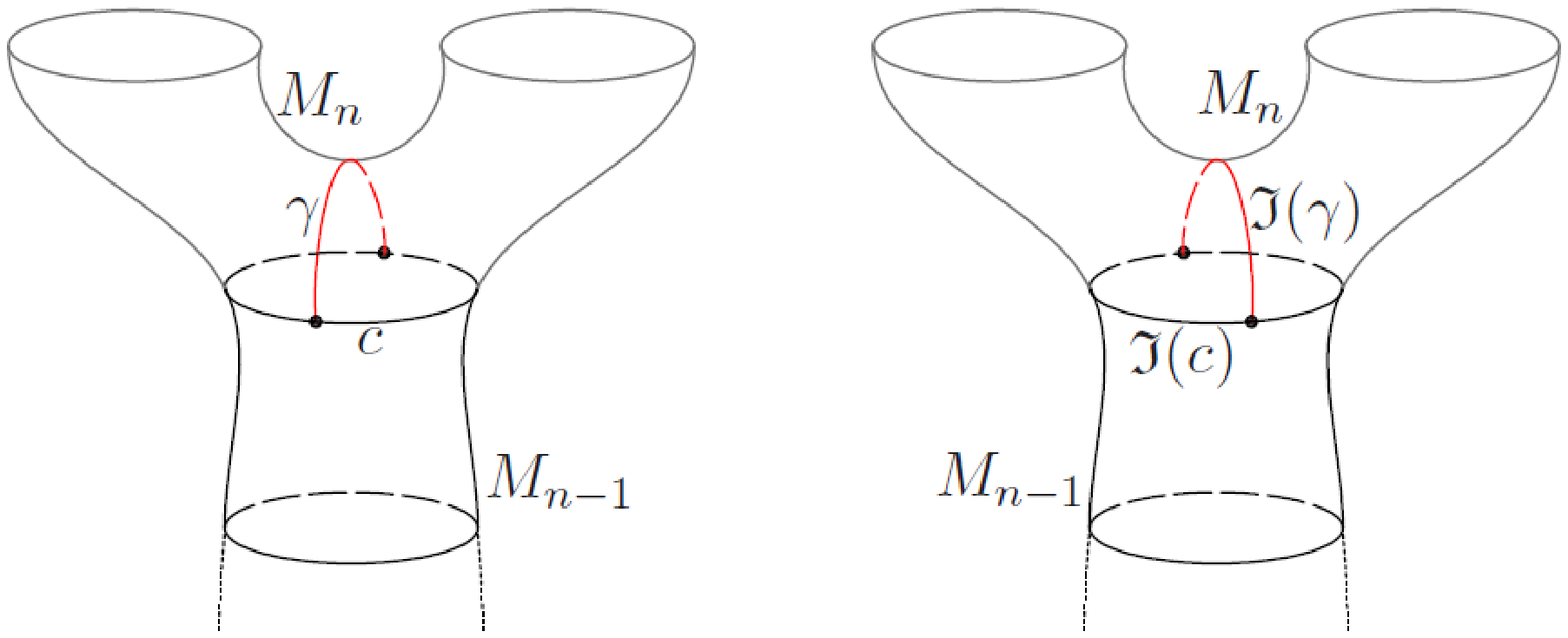}}
        \end{center}
        \vspace{-0.25cm}
\caption{$M_n$; procedure 2.}\label{fig:topo1}
\end{figure}

\item $M_n$ is a tubular neighborhood of $M_{n-1}\cup \gamma\cup\I(\gamma)$, where $\gamma$ is a Jordan arc in $\Ncal$ with an endpoint in a connected component $c$ of $b M_{n-1}$, the other endpoint in $\I(c)$, and otherwise disjoint from $M_{n-1}$, such that $\gamma\cap\I(\gamma)=\emptyset$ and $M_{n-1}\cup \gamma\cup\I(\gamma)$ is an $\I$-admissible subset in $\Ncal$. In this case, $M_n$ has the same number of boundary components as $M_{n-1}$ and one more handle; hence $\chi(M_n^\circ\setminus M_{n-1})=-2$. (See Fig.\ \ref{fig:topo2}.)
\begin{figure}[ht]
    \begin{center}
    \scalebox{0.50}{\includegraphics{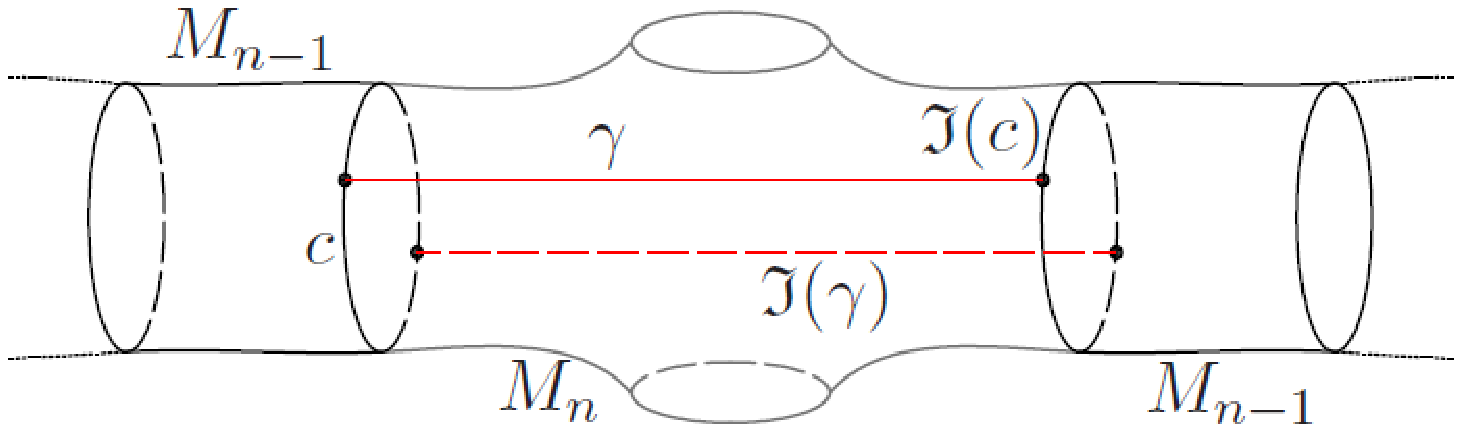}}
        \end{center}
        \vspace{-0.25cm}
\caption{$M_n$; procedure 3.}\label{fig:topo2}
\end{figure}

\item $M_n$ is a tubular neighborhood of $M_{n-1}\cup \gamma\cup\I(\gamma)$, where $\gamma$ is a Jordan arc in $\Ncal$ with an endpoint in a connected component $c_1$ of $b M_{n-1}$, the other endpoint in a connected component $c_2\neq c_1$ of $b M_{n-1}$, and otherwise disjoint from $M_{n-1}$, such that $c_2\cap \I(c_1)=\emptyset$, $\gamma\cap\I(\gamma)=\emptyset$, and $M_{n-1}\cup \gamma\cup\I(\gamma)$ is an $\I$-admissible subset in $\Ncal$. In this case, $M_n$ has two less boundary components than $M_{n-1}$ and two more handles; hence $\chi(M_n^\circ\setminus M_{n-1})=-2$. (See Fig.\ \ref{fig:topo3}.)
\begin{figure}[ht]
    \begin{center}
    \scalebox{0.50}{\includegraphics{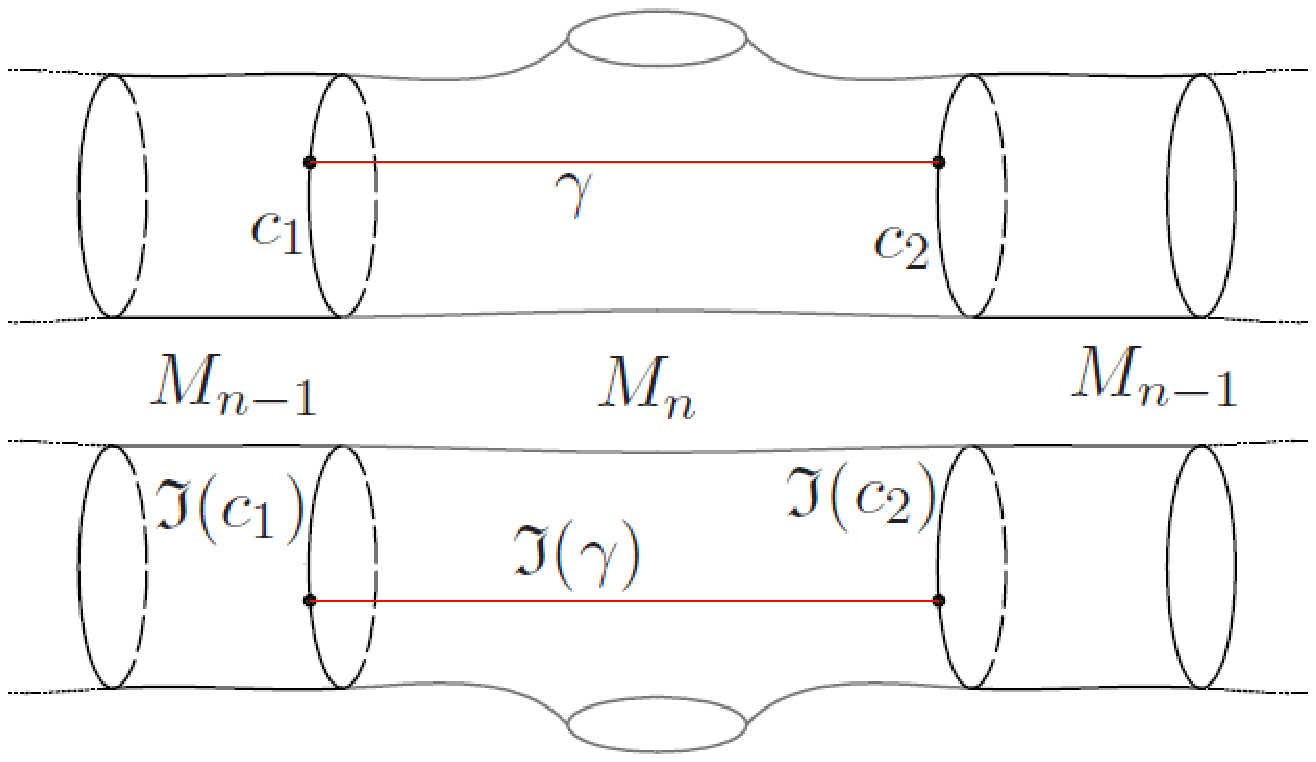}}
        \end{center}
        \vspace{-0.25cm}
\caption{$M_n$; procedure 4.}\label{fig:topo3}
\end{figure}
\end{enumerate}
\end{remark}

Let $0<\xi<\epsilon$ to be specified later and let us construct a sequence $\{F_n=(F_{n,j})_{j=1,2,3}\}_{n\in\n}$, where $F_n\in\Mcal_\I(M_n)$, such that
\begin{enumerate}[\rm (A{$_n$})]
\item $\|F_n-F_{n-1}\|_{1,M_{n-1}}<\xi/2^{n+1}$,
\item  $\partial F_{n,3}$ vanishes nowhere on $M_n\setminus S$, $(\partial F_{n,3}|_E)=(\phi_3|_E)\in\div_\I(R_S)$  for any connected component $E$ of $R_S$ such that $\phi_3$ does not vanish everywhere on $E$, and
\item $\pgot_{F_n}=\pgot|_{\Hcal_1(M_n,\z)}$ for all $n\in\n$.
\end{enumerate}

We follow a recursive process. Set $F_1:=F_{0}(P_0)+\Re \int_{P_0} \Phi_1$, where $\Phi_1\in\Omega_{\hgot,\I}(M_1)^3$ is a triple resulting to apply Lemma \ref{lem:weiaprox} to the data
\[
S=M_{0},\enskip \text{$W$ a tubular neighborhood of $M_1$},\enskip \Phi=\partial F_{0}=\Phi_0,
\]
close enough to $\partial F_{0}$ in the $\Ccal^0$ topology on $M_{0}$ to insure {\rm (A$_1$)}. Recall that $M_1$ is a tubular neighborhood of $M_0$, hence $\Phi_1$ has no real periods and $F_1$ is well defined. Properties {\rm (B$_1$)} and {\rm (C$_1$)} follow from {\rm (a)}--{\rm (d)}.

Let $n\geq 2$, assume that we have constructed $F_1,\ldots, F_{n-1}$, and let us furnish $F_n$. We distinguish the following two cases:

\noindent {\it Case $1$.} Assume that $\chi(M_n^\circ\setminus M_{n-1})=0$; hence $M_n^\circ\setminus M_{n-1}$ consists of a finite family of pairwise disjoint open annuli and $\Hcal_1(M_n,\z)=\Hcal_1(M_{n-1},\z)$; see Remark \ref{rem:topology}-1. In this case we take $F_n:=F_{n-1}(P_0)+\Re \int_{P_0} \Phi_n$, where $\Phi_n\in\Omega_{\hgot,\I}(M_n)^3$ is a triple resulting to apply Lemma \ref{lem:weiaprox} to the data
\[
S=M_{n-1},\enskip \text{$W$ a tubular neighborhood of $M_n$},\enskip \Phi=\partial F_{n-1},
\]
close enough to $\partial F_{n-1}$ in the $\Ccal^0$ topology on $M_{n-1}$ to ensure {\rm (A$_n$)}. Since $M_n$ is a tubular neighborhood of $M_{n-1}$, then $\Phi_n$ has no real periods and $F_n$ is well defined. Properties   {\rm (B$_n$)} and  {\rm (C$_n$)} follow straightforwardly from  {\rm (B$_{n-1}$)}, {\rm (C$_{n-1}$)}, and Lemma \ref{lem:weiaprox}.

\noindent {\it Case $2$.} Assume that $\chi(M_n^\circ\setminus M_{n-1})=-2$. In this case, there exists an analytical Jordan arc $\gamma\subset M_n^\circ\setminus M_{n-1}^\circ$, attached to $b M_{n-1}$ at its endpoints and otherwise disjoint to $M_{n-1}$, such that $\gamma\cap\I(\gamma)=\emptyset$, $\tilde S:=M_{n-1}\cup\gamma\cup\I(\gamma)$ is an $\I$-admissible set in $\Ncal$, and $\chi(M_n^\circ \setminus \tilde S)=0$; see Remark \ref{rem:topology}-2,3,4. Extend $F_{n-1}$ to a generalized marked immersion $\tilde F\in\Mcal_{\ggot,\I}^*(\tilde S)$ such that $\pgot_{\tilde F}=\pgot|_{\Hcal_1(\tilde S,\z)}$. Up to approximating $\tilde{F}$ by a minimal immersion in $\Mcal_\I(\tilde M_{n-1})$  via Lemma \ref{lem:weiaprox}, where $\tilde M_{n-1}\subset M_n^\circ$ is a tubular neighborhood of $\tilde{S}$,  one can reduce the proof to the previous case. 

This concludes the construction of the sequence $\{F_n\}_{n\in\n}$.

By properties {\rm (A$_{n}$)}, $n\in \n$, the sequence $\{F_n\}_{n\in\n}$ converges in the $\Ccal^1$ topology on compact sets of $\Ncal$ to an $\I$-invariant conformal harmonic map $Y=(Y_j)_{j=1,2,3}\colon \Ncal\to\r^3$ such that $\|Y-X_\varpi\|_{1,S}<\xi<\epsilon$; take also {\rm (c)} into account. From {\rm (b)}, {\rm (B$_{n}$)}, $n\in \n$, and Hurwitz's theorem, it follows that $\partial Y_3$ vanishes nowhere on $\Ncal\setminus S$ and 
 $(\partial Y_{3}|_E)=(\phi_3|_E)\in\div_\I(R_S)$  for any connected component $E$ of $R_S$ such that $\phi_3$ does not vanish everywhere on $E$. On the other hand, if $\xi$ is taken small enough from the beginning, then $Y$ is a conformal minimal immersion; indeed, it has branch points neither on $S$ (since $\partial Y$ is close to $\partial X_\varpi$ on $S$) nor on $\Ncal\setminus S$ (since $\partial Y_3$ vanishes nowhere on $\Ncal\setminus S$). Finally, {\rm (d)} and {\rm (C$_{n}$)}, $n\in \n$, give $\pgot_Y=\pgot$. This proves statement {\rm (I)}.

In order to prove statement {\rm (II)} we reason analogously but using Lemma \ref{lem:weiaprox2} instead of Lemma \ref{lem:weiaprox}.
\end{proof}

Notice that Theorem \ref{th:Mergelyan} is a general existence result of non-orientable minimal surfaces in $\r^3$ with arbitrary conformal structure. In fact, in the next section we construct such surfaces with additional geometrical properties; see Theorems \ref{th:proper} and \ref{th:coordenada}.

%
%
%


\section{Applications}\label{sec:apli}

We conclude the paper with some applications of the results in the previous section. In Subsec.\ \ref{subsec:RM} we will derive approximation theorems of Runge-Mergelyan's type for other objects than non-orientable minimal surfaces; see Corollary \ref{co:null} and Theorem \ref{th:harmonic}. In Subsec.\ \ref{subsec:GN} we will prove an existence theorem of Gunning-Narasimhan's type on non-orientable Riemann surfaces (see Theorem \ref{th:GN}). Finally, in Subsec.\ \ref{subsec:proper} and \ref{subsec:coordenada} we show general existence results of non-orientable minimal surfaces in $\r^3$ with given underlying conformal structure and additional topological or geometric properties.

\subsection{Some Runge-Mergelyan's type results}\label{subsec:RM}

A holomorphic immersion $(F_j)_{j=1,2,3} \colon \Ncal\to\c^3$ is said to be a {\em null curve} if $\sum_{j=1}^3 (dF_j)^2$ vanishes everywhere on the open Riemann surface $\Ncal$. Minimal surfaces in $\r^3$ are locally the real part of null curves in $\c^3$. (See \cite{Osserman-book} for a good reference.)

We can now derive the analogous result to Theorem \ref{th:Mergelyan} for {\em $\I$-symmetric} null curves.
\begin{corollary}[Runge-Mergelyan's Theorem for $\I$-symmetric null curves in $\c^3$]\label{co:null}
Let $S\subset\Ncal$ be an $\I$-admissible subset and let $F=(F_j)_{j=1,2,3}\colon S\to\c^3$ be a smooth function in $\Fcal_{\hgot,\I}(S)^3$ such that $\sum_{j=1}^3 (dF_j)^2$ vanishes everywhere on $S$ and $\sum_{j=1}^3 |dF_j|^2$ vanishes nowhere on $S$.

The following assertions hold:
\begin{itemize}
\item $F$ can be uniformly approximated in the $\Ccal^1$ topology on $S$ by null curves $H=(H_j)_{j=1,2,3}\colon \Ncal\to\c^3$ in $\Fcal_{\hgot,\I}(\Ncal)^3$ such that $dH_3$ vanishes nowhere on $\Ncal\setminus R_S$.
Furthermore,  $H$ can be chosen so that $(d H_3|_E)=(dF_3|_E)\in\div(E)$ for any connected component $E$ of $R_S$ such that $dF_3$ does not vanish everywhere on $E$.

\item If $F_3$ is non-constant and extends to $\Ncal$ as a holomorphic function whose differential vanishes nowhere on $C_S$, then $F$ can be approximated in the $\Ccal^1$ topology on $S$ by  null curves $H=(H_j)_{j=1,2,3}\colon \Ncal\to\c^3$ in $\Fcal_{\hgot,\I}(\Ncal)^3$, where $H_3=F_3$.
\end{itemize}
\end{corollary}
\begin{proof} Up to suitably extending $F$ to a connected $\I$-admissible subset of $\Ncal$ containing $S$ (see Claim \ref{cla:con}), we assume without loss of generality that $S$ is connected.
 
By Theorem \ref{th:Mergelyan}, there exists a sequence $\{Y_n=(Y_{n,j})_{j=1,2,3}\}_{n\in \n} \subset \Mcal_\I(\Ncal)$ with $\pgot_{Y_n}=0$ for all $n$, approximating $X_\varpi\equiv (X,\varpi)=(\Re(F), \Im(dF))$ in the $\Ccal^1$ topology on $S$, and whose third coordinates $Y_{n,3}$, $n\in \n$, satisfy the required properties. If we fix $P_0\in S$, the sequence of null curves $\{F(P_0)+\int_{P_0} \partial Y_n\}_{n\in \n}$ on $\Ncal$ proves the Corollary.
\end{proof}

We next point out that Theorem \ref{th:Mergelyan} is also valid for marked harmonic functions in the following sense:
\begin{definition}\label{def:markedh}
Let $S\subset\Ncal$ be an $\I$-admissible subset. By a marked $\I$-invariant harmonic function on $S$ we mean a couple $h_\theta\equiv(h,\theta)$, where $h\colon S\to\r^3$ is an $\I$-invariant $\Ccal^1$ function, harmonic on $R_S$, and $\theta\in\Omega_{\hgot,\I}(S)$ is a $1$-form such that $\theta|_{R_S}$ equals the complex derivative $\partial (h|_{R_S})$ of $h|_{R_S}$, $\theta$ has no real periods, and $h=\Re\int^P \theta$. 

If $h_\theta$ is a marked $\I$-invariant harmonic function we denote by $\partial h_\theta=\theta$. Analogously to Def.\ \ref{def:C1-top}, the space of marked $\I$-invariant harmonic functions on $S$ is endowed with a natural $\Ccal^1$ topology.
\end{definition}

If $\theta \in \Omega_{\hgot,\I}(S)$, $\gamma \in \Hcal(S,\z)$, and $\I_*(\gamma)=\gamma$, then $\int_\gamma \theta \in \r$. Likewise, if $\I_*(\gamma)=-\gamma$ then $\int_\gamma \theta \in \imath\r$. In particular, $\theta$ has no real periods if and only if 
$\int_\gamma \theta =0$ for all $\gamma \in \Hcal(S,\z)$ with $\I_*(\gamma)=\gamma$, and in this case $\int_{\I_*(\gamma)} \theta=-\int_\gamma \theta\in \imath \r$ for all  $\gamma \in \Hcal(S,\z)$.

\begin{theorem}[Runge-Mergelyan's Theorem for harmonic functions of non-orientable Riemann surfaces]\label{th:harmonic}
Let $S\subset\Ncal$ be an $\I$-admissible subset. Let $h_\theta$ be a marked  $\I$-invariant harmonic function on $S$. Let $\pgot\colon \Hcal_1(\Ncal,\z)\to\r$ be a group homomorphism such that $\pgot(\I_*(\gamma))=-\pgot(\gamma)$ for all $\gamma\in \Hcal_1(\Ncal,\z)$, and  $\pgot(\gamma)=\Im\int_\gamma \partial h_\theta$ for all $\gamma\in\Hcal_1(S,\z)$.

Then $h_\theta$ can be approximated in the $\Ccal^1$ topology on $S$ by $\I$-invariant harmonic functions $\hat h$ on $\Ncal$, satisfying that $\partial \hat h$ vanishes nowhere on $\Ncal\setminus R_S$ and $\pgot(\gamma)=\Im\int_\gamma \partial \hat h$ for all $\gamma\in\Hcal_1(\Ncal,\z)$. Furthermore, $\hat h$ can be chosen so that $(\partial\hat h|_E)=(\partial h_\theta|_E)\in\div_\I(E)$ for any connected component $E$ of $R_S$ such that   $h|_E$ is non-constant.
\end{theorem}
\begin{proof}
We assume without loss of generality that $\partial h_\theta$ vanishes nowhere on $C_S$ (cf.\ Claim \ref{ass:motion}) and that $S$ is connected (cf.\ Claim \ref{cla:con} and the proof of Theorem \ref{th:Mergelyan}).
Denote by $\phi_3=\partial h_\theta\in\Omega_{\hgot,\I}(S)$. Let $U\subset \Ncal$ be an $\I$-invariant relatively compact domain with finite topology, containing $S$. Reasoning as in the proof of Lemma \ref{lem:gaprox}, the Riemann-Roch theorem gives $g_0\in\Gcal_\I(U)$ such that  $(\phi_3|_{R_S})=(g_0)_0 (g_0)_\infty$. Then one can easily extend $g_0|_{R_S}$ to a function $g\in\Gcal_\I(S)$ such that the triple $\Phi=(\phi_j)_{j=1,2,3}$, obtained from the couple $(g,\phi_3)$ via \eqref{eq:recoverX}, satisfies the requirements in Lemma \ref{lem:weiaprox}.

Let $M_0:=S$ and let $\{M_n\}_{n\in\n}$ be an exhaustion of $\Ncal$ by Runge connected $\I$-invariant compact regions such that the Euler characteristic $\chi(M_n^\circ\setminus M_{n-1})\in\{0,-2\}$ for all $n\in\n$; see Remark \ref{rem:topology}. 

Let $\epsilon>0$. Fix $\xi>0$ to be specified later. Arguing as in the proof of Theorem \ref{th:Mergelyan}, a recursive application of Lemma \ref{lem:weiaprox} gives a sequence $\{\Psi_n=(\psi_{n,j})_{j=1,2,3}\}_{n\in\n}\subset\Omega_{\hgot,\I}(M_n)^3$ such that $\psi_{n,3}$ vanishes nowhere on $M_n\setminus R_S$, $\int_\gamma \psi_{n,3}=\imath \pgot(\gamma)$ for all $\gamma\in\Hcal_1(M_n,\z)$, and $\{\Psi_n\}_{n\in\n}$ converges in the $\Ccal^0$ topology on compact sets of $\Ncal$ to a triple $\Psi=(\psi_j)_{j=1,2,3}\in\Omega_{\hgot,\I}(\Ncal)^3$ with 
\begin{equation}\label{eq:epxi}
\|\Psi-\Phi\|_{0,M_0}<\xi.
\end{equation}
Therefore $\psi_3$ vanishes nowhere on $\Ncal\setminus R_S$ (by Hurwitz's theorem) and $\int_\gamma \psi_3=\imath \pgot(\gamma)$ for all $\gamma\in\Hcal_1(\Ncal,\z)$. Furthermore,  $\psi_{n,3}$ can be chosen so that $(\psi_{n,3}|_E)=(\phi_3|_E)\in\div_\I(E)$ for any connected component $E$ of $R_S$ such that   $\phi_3$ does not vanish everywhere on $E$, $n\in \n$. Therefore, Hurwitz's theorem gives that $(\psi_3|_E)=(\phi_3|_E)\in\div_\I(E)$ as well for any such $E$.

By \eqref{eq:epxi}, the $\I$-invariant harmonic function $\hat h:=h(P_0)+\Re\int_{P_0} \psi_3$, $P_0\in S$, satisfies $\|\hat h-h_\theta\|_{1,S}<\epsilon$, provided that $\xi$ is chosen small enough. This proves the theorem.
\end{proof}

\subsection{An application to Riemann surface theory}\label{subsec:GN}

Gunning and Narasimhan \cite{GunningNarasimhan} showed that every open Riemann surface carries exact nowhere vanishing holomorphic $1$-forms. This result was extended to the existence of holomorphic $1$-forms with prescribed periods and divisor by Kusunoki and Sainouchi \cite{KusunokiSainouchi}. Let us show the analogous result for non-orientable Riemann surfaces.

\begin{theorem}\label{th:GN}
Let $D'$ be an integral divisor on $\Ncal$, possibly with {\em countably infinite support}, such that ${\rm supp}(D')\cap{\rm supp}(\I(D'))=\emptyset$ and ${\rm supp}(D')\cap K$ is finite for any compact $K\subset \Ncal$. Call $D=D'\cup\I(D')$.
Let $\pgot\colon \Hcal_1(\Ncal,\z)\to\c^3$ be a group homomorphism such that $\pgot(\I_*(\gamma))=\overline{\pgot(\gamma)}$ for all $\gamma\in\Hcal_1(\Ncal,\z)$.
 
Then there exists $\vartheta\in \Omega_{\hgot,\I}(\Ncal)$ such that $(\vartheta)=D$ and $\int_\gamma \vartheta =\pgot(\gamma)$ for all $\gamma\in\Hcal_1(\Ncal,\z)$.
\end{theorem}
\begin{proof}
Let $M_0\subset\Ncal$ be a connected $\I$-admissible set such that  $R_{M_0}=U\cup \I(U)$ where $U$ is a closed disc in $\Ncal$ and $U \cap \I(U)=\emptyset$, and ${\rm supp}(D)\cap (M_0\setminus M_0^\circ)=\emptyset$. Denote by $D_0$ the restriction of $D$ to $M_0$ (that is, the unique integer divisor in $M_0$ such that ${\rm supp}(D/D_0)\cap M_0=\emptyset$), and recall that ${\rm supp}(D_0)$ consists of finitely many points. 
Take  $\phi_{0,3}\in \Omega_{\hgot,\I}(R_{M_0})$ vanishing everywhere on no connected component of $R_{M_0}$ and with $(\phi_{0,3})=D_0$. As in the proof of Theorem \ref{th:Mergelyan}, one can extend $\phi_{0,3}$ to a smooth $1$-form $\psi_0\in \Omega_{\hgot,\I}(M_0)$ vanishing nowhere on $M_0 \setminus M_0^\circ$ and  satisfying that $\int_\gamma \psi_0=\pgot(\gamma)$ for all $\gamma\in \Hcal_1(M_0,\z)$.

Let  $\{M_n\}_{n\in\n}$ be an exhaustion of $\Ncal$ by Runge connected $\I$-invariant compact regions such that the Euler characteristic $\chi(M_n^\circ\setminus M_{n-1})\in\{0,-2\}$ and ${\rm supp}(D)\cap bM_n=\emptyset$ for all $n\in\n$; see Remark \ref{rem:topology}. Denote by $D_n$ the restriction of $D$ to $M_n$, and recall that ${\rm supp}(D_n)$ consists of finitely many points, $n\in\n$.

Assume that $\chi(M_1^\circ\setminus M_0)=0$. Let $N_0\subset M_1^\circ\setminus M_0$ be a  Runge $\I$-invariant compact region containing ${\rm supp}(D_1)\setminus {\rm supp}(D_0)$, and consisting of a finite collection of pairwise disjoint closed discs. Notice that  $M_0\cup N_0$ is $\I$-admissible. Take $\phi_{1,3}\in\Omega_{\hgot,\I}(M_0\cup N_0)$ vanishing everywhere on no connected component of $R_{M_0}\cup N_0$, and satisfying that $\phi_{1,3}|_{M_0}=\psi_0$, $(\phi_{1,3})=D_1$, and $\int_\gamma \phi_{1,3}=\pgot(\gamma)$ for all $\gamma\in \Hcal_1(M_1,\z)=\Hcal_1(M_0,\z)$. As in the proof of Theorem \ref{th:harmonic}, one can easily find a triple $\Phi_1=(\phi_{1,j})_{j=1,2,3}\in \Omega_{\hgot,\I}(M_0\cup N_0)^3$ meeting the requirements of Lemma \ref{lem:weiaprox}. Therefore, this lemma furnishes $\Psi_1=(\psi_{1,j})_{j=1,2,3}\in \Omega_{\hgot,\I}(M_1)^3$ as close as desired to $\Phi_1$ in the $\Ccal^0$ topology on $M_0\cup N_0$, satisfying that $\Psi_1-\Phi_1$ is exact on $M_0\cup N_0$ and $(\psi_{1,3})=(\phi_{1,3}).$ Call $\psi_1:=\psi_{1,3}\in \Omega_{\hgot,\I}(M_1)$ and notice that $(\psi_1)=D_1$, $\int_\gamma \psi_1=\pgot(\gamma)$ for all $\gamma\in \Hcal_1(M_1,\z)$, and $\psi_1$ is as close as desired to $\psi_0$ in the $\Ccal^0$ topology on $M_0$.

Assume that, on the contrary, $\chi(M_1^\circ\setminus M_0)=-2$. Then take a Jordan arc $\alpha$ in $M_1^\circ\setminus (M_0^\circ \cup {\rm supp}(D))$, with endpoints in $R_{M_0}$ and otherwise disjoint from $M_0$, such that $S:=M_0\cup\alpha\cup\I(\alpha)$ is $\I$-admissible and $\chi(M_1^\circ\setminus S)=0$; see Remark \ref{rem:topology}. Extend (with the same name) $\phi_{3,0}$ to a smooth $1$-form $\phi_{3,0}\in\Omega_{\hgot,\I}(S)$ vanishing nowhere on $\alpha$ and satisfying that $\int_\gamma \phi_{3,0}=\pgot(\gamma)$ for all $\gamma\in \Hcal_1(M_1,\z)=\Hcal_1(S,\z)$. Then, one can follow the argument in the above paragraph, replacing $M_0$ by $S$, and obtain as above $\psi_1\in \Omega_{\hgot,\I}(M_1)$ as close as desired to $\psi_0$ in the $\Ccal^0$ topology on $M_0$, with $(\psi_1)=D_1$ and $\int_\gamma \psi_1=\pgot(\gamma)$ for all $\gamma\in \Hcal_1(M_1,\z)$.

Repeating this argument inductively, one constructs a sequence $\{\psi_n\}_{n\in\n}\subset\Omega_{\hgot,\I}(M_n)$ such that $\psi_n$ is as close as desired to $\psi_{n-1}$ in the $\Ccal^0$ topology on $M_{n-1}$ for all $n>1$, $(\psi_n)=D_n$ and $\int_\gamma \psi_n=\pgot(\gamma)$ for all $\gamma\in \Hcal_1(M_n,\z)$, for all $n\in\n$. Furthermore, one can assume that $\{\psi_n\}_{n\in\n}$ converges uniformly on compact subsets of $\Ncal$ to a holomorphic $1$-form $\vartheta\in \Omega_{\hgot,\I}(\Ncal)$. Since $\psi_0$ is not identically zero and $\vartheta$ can be constructed as close as desired to $\psi_0$ on $M_0$, then $\vartheta$ may be assumed to be non-identically zero as well. Obviously $\int_\gamma \vartheta =\pgot(\gamma)$ for all $\gamma\in\Hcal_1(\Ncal,\z)$. By Hurwitz's theorem, $(\vartheta)=D$ and we are done. 
\end{proof} 


\subsection{Non-orientable minimal surfaces in $\r^3$ properly projecting into $\r^2$}\label{subsec:proper}

In this subsection we show that any open Riemann surface $\Ncal$ endowed with an antiholomorphic involution $\I\colon\Ncal\to\Ncal$ without fixed points, is furnished with an $\I$-invariant conformal minimal immersion $\Ncal\to\r^3$ whose image surface is a non-orientable minimal surface properly projecting into a plane, contained in a wedge in $\r^3$ of any given angle greater than $\pi$. Furthermore, the flux map of such surface can be prescribed under the compatibility condition \eqref{eq:flux-non}. This existence theorem links with a classical question by Schoen and Yau \cite[p.\ 18]{SchoenYau-harmonic}; see \cite{AL-proper,AL-conjugada} for a good setting on this problem.

\begin{theorem}\label{th:proper}
Let $\pgot\colon \Hcal_1(\Ncal,\z) \to \r^3$ be a group homomorphism satisfying
\[
\pgot(\I_*(\gamma))=-\pgot(\gamma)\enskip \forall \gamma\in\Hcal_1(\Ncal,\z),
\] 
and let $\theta$ be a real number in  $(0,\pi/2)$.

Let $M \subset \Ncal$ be a Runge $\I$-invariant compact region, and consider $Y\in \Mcal_\I(M)$ with flux map $\pgot_Y=\pgot|_{\Hcal_1(M,\z)}$, satisfying
\[
(x_3+\tan (\theta) |x_1|) \circ Y>1\enskip \text{everywhere on $M$.}
\]

Then for any $\epsilon > 0$ there exists a conformal minimal immersion $X\colon \Ncal \to \r^3$ satisfying $X\circ\I=X$ and the following properties:
\begin{itemize}
  \item $\pgot_X=\pgot,$
  \item $(x_3+\tan (\theta) |x_1|)\circ X\colon \Ncal\to\r$ is a positive proper function, and
  \item $\|X-Y\|_{1,M}< \epsilon.$
\end{itemize}
\end{theorem}

The corresponding theorem for orientable minimal surfaces was obtained by the authors in \cite[Theorem 5.6]{AL-proper}, as application of the Runge-Mergelyan approximation result \cite[Theorem 4.9]{AL-proper}.
We adapt the proof in \cite{AL-proper} to the  non-orientable framework, sketching the necessary modifications. In this case our main tool is Theorem \ref{th:Mergelyan}. The complete details could easily be filled in by an interested reader.

We denote by $x_k\colon\r^3 \to \r$ the $k$-th coordinate function, $k=1,2,3.$ Given numbers $\theta\in(-\pi/2,\pi/2)$ and $\delta \in \r,$ we denote by 
\[
\Pi_\delta (\theta)=\big\{(x_1,x_2,x_3) \in \r^3 \colon x_3+\tan(\theta) x_1 > \delta\big\}.
\]

Theorem \ref{th:proper} will follow from a standard recursive application of the following approximation result.
\begin{lemma} \label{lem:fun}
Let $M,$ $V \subset \Ncal$ be two Runge $\I$-invariant compact regions with analytical boundary such that $M \subset V^\circ$ and the Euler characteristic $\chi(V\setminus M^\circ)\in\{-2,0\}$.
 
Let $X \in \Mcal_\I(M)$ and let $\pgot\colon \Hcal_1(V,\z)\to\r$ be any homomorphism extension of the flux map $\pgot_X$ of $X$, satisfying
\[
\pgot(\I_*(\gamma))=-\pgot(\gamma)\enskip \forall \gamma\in\Hcal_1(\Ncal,\z).
\] 
Let  $\theta\in (0,\pi/4)$ and $\delta>0$, and assume that 
\begin{equation}\label{eq:lemma}
X(bM) \subset \Pi_\delta(\theta) \cup \Pi_\delta(-\theta).
\end{equation}

Then, for any $\epsilon>0$  there exists $Y \in \Mcal_\I(V)$ enjoying the following properties:
\begin{enumerate}[\rm (i)]
\item The flux map $\pgot_Y$ of $Y$ equals $\pgot.$ 
\item $\|Y-X\|_{1,M}<\epsilon.$  
\item $Y(bV) \subset \Pi_{\delta+1}(\theta) \cup \Pi_{\delta+1}(-\theta).$ 
\item $Y(V\setminus M) \subset \Pi_\delta(\theta) \cup \Pi_\delta(-\theta).$
\end{enumerate}
\end{lemma}
\begin{proof}[Proof of Lemma \ref{lem:fun} in case $\chi(V\setminus M^\circ)=0$.]
Since $M \subset V^\circ$ and $V^\circ\setminus M$ has no relatively compact connected components in $V^\circ,$ then
 $V\setminus M^\circ=\cup_{j=1}^\jgot (A_j \cup \I(A_j)),$ where $\jgot\in\n$ denotes the number of boundary components of $V$ (hence, of $M$) and $A_1,\I(A_1),\ldots, A_\jgot,\I(A_\jgot)$ are pairwise disjoint compact annuli. 
 
Write $bA_j=\alpha_j \cup \beta_j,$ where $\alpha_j\subset b M$ and $\beta_j\subset b V$  for all $j=1,\ldots,\jgot.$ Obviously, $\I(\alpha_j)\subset b M$, $\I(\beta_j)\subset b V$, and $b \I(A_j)=\I(\alpha_j) \cup \I(\beta_j)$ for all $j=1,\ldots,\jgot.$

From inclusion \eqref{eq:lemma}, it follows the existence of a natural number $\igot\geq 2$, a collection of sub-arcs $\big\{\alpha_j^i \colon (i,j)\in I=\z_\igot\times \{1,\ldots,\jgot\}\big\}$, where $\alpha_j^i\subset\alpha_j$ for all $(i,j)\in I$ and $\z_\igot=\{0,\ldots,\igot-1\}$ denotes the additive cyclic group of integers modulus $\igot\in\n$, and subsets $I_+$ and $I_-$ of $I$, such that
\begin{itemize}
\item $I_+\cap I_-=\emptyset,$ $I_+\cup I_-=I$,
\item $\alpha_j^i$ and $\alpha_j^{i+1}$ have a common endpoint $Q_j^{i+1}$ and are otherwise disjoint, and
\item $X(\alpha_j^i) \subset \Pi_\delta(\pm\theta)$ for all $(i,j)\in I_\pm$.
\end{itemize}
In particular, $\alpha_j=\cup_{i\in \z_\igot} \alpha_j^i$ for all $j=1,\ldots,\jgot$, and $X(\I(\alpha_j^i)) \subset \Pi_\delta(\pm\theta)$ for all $(i,j)\in I_\pm$; recall that $X$ is $\I$-invariant.

Choose a family $\{r_j^i\colon (i,j)\in I\}$ of pairwise
disjoint analytical compact Jordan arcs such that $r_j^i$ is contained in $A_j$, has
initial point $Q_j^i \in \alpha_j,$ final point $P_j^i\in \beta_j,$ is otherwise disjoint from $bA_j,$ and meets transversally $\alpha_j^i$ at the point $Q_j^i,$ for all $(i,j)\in I$.
The set 
\[
S:=M \cup \big( \cup_{(i,j)\in I} r_j^i\cup \I(r_j^i) \big)
\]
is $\I$-admissible in the sense of Def.\ \ref{def:admi}.

In a first step, we construct an $\I$-invariant conformal minimal immersion $H\in\Mcal_\I(V)$ meeting the theses of the lemma on points of $S;$ more specifically, satisfying
\begin{enumerate}[\rm (1$_H$)]
\item $\|H-X\|_{1,S}<\epsilon/3,$
\item $H(r_j^i\cup \alpha_j^i\cup r_j^{i+1})\subset \Pi_\delta(\pm\theta)$ for all $(i,j)\in I_\pm$, 
\item $H(\{P_j^i,P_j^{i+1}\})\subset \Pi_{\delta+1}(\pm\theta)$ for all $(i,j)\in I_\pm$, and
\item $\pgot_H=\pgot$.
\end{enumerate}
Such an $H$ is furnished by Theorem \ref{th:Mergelyan}-{\rm (I)} applied to a suitable $\I$-invariant extension $\hat X$ of $X$ to $S$, formally meeting properties {\rm (2$_H$)} and {\rm (3$_H$)} (cf. \cite[Subsec.\ 5.1]{AL-proper}). To construct such an extension, we first define $\hat X$ over $\cup_{(i,j)\in I}r_j^i$ and then we extend it to $S$ (that is to say; we define $\hat X$ over $\cup_{(i,j)\in I}\I(r_j^i)$) to be $\I$-invariant.

Denote by $\Omega_j^i$ the closed disc in $A_j$
bounded by $\alpha_j^i\cup r_j^i \cup r_j^{i+1}$ and the compact Jordan arc $\beta_j^i\subset\beta_j$ connecting $P_j^i$ and
$P_j^{i+1},$ and containing no $P_j^k$ for $k\neq i,i+1,$ $(i,j)\in I$. Since $H$ is continuous, then properties {\rm (2$_H$)} and {\rm (3$_H$)} extend to small open neighborhoods of $r_j^i\cup \alpha_j^i\cup r_j^{i+1}$ and $\{P_j^i,P_j^{i+1}\}$, respectively; hence there exists a closed disc $K_j^i\subset \Omega_j^i\setminus (r_j^i\cup \alpha_j^i\cup r_j^{i+1})$, intersecting $\beta_j^i$ in a compact Jordan arc, such that
\begin{itemize}
\item $H(\overline{\Omega_j^i\setminus K_j^i})\subset \Pi_\delta(\pm\theta)$ for all $(i,j)\in I_\pm$, and
\item $H(\overline{\beta_j^i\setminus K_j^i})\subset \Pi_{\delta+1}(\pm\theta)$ for all $(i,j)\in I_\pm$.
\end{itemize}

Assume without loss of generality that $I_+\neq\emptyset$; otherwise $I_-=I\neq\emptyset$ and we would reason in a symmetric way. Consider the $\I$-admissible set 
\[
S_+:=M \cup \big( \cup_{(i,j)\in I_-} \Omega_j^i\cup \I(\Omega_j^i) \big)\cup \big( \cup_{(i,j)\in I_+} K_j^i\cup \I(K_j^i) \big).
\]

In a second step, we construct an $\I$-invariant conformal minimal immersion $Z\in\Mcal_\I(V)$ meeting the theses of the lemma on points of $S_+;$ more concretely, satisfying
\begin{enumerate}[\rm (1$_Z$)]
\item $\|Z-H\|_{1,M \cup ( \cup_{(i,j)\in I_-} \Omega_j^i\cup \I(\Omega_j^i))}<\epsilon/3,$
\item $Z(\overline{\Omega_j^i\setminus K_j^i})\subset \Pi_\delta(\pm\theta)$ for all $(i,j)\in I_\pm$, 
\item $Z(\overline{\beta_j^i\setminus K_j^i})\subset \Pi_{\delta+1}(\pm\theta)$ for all $(i,j)\in I_\pm$,
\item $Z(K_j^i)\subset \Pi_{\delta+1}(-\theta)$ for all $(i,j)\in I_+=I\setminus I_-$, and
\item $\pgot_Z=\pgot$.
\end{enumerate}
The immersion $Z$ is furnished by Theorem \ref{th:Mergelyan}-{\rm (II)} applied to a suitable $\I$-invariant extension $\hat H$ of $H|_{M \cup (\cup_{(i,j)\in I_-} \Omega_j^i\cup \I(\Omega_j^i))}$ to $S_+$. The key point here is to insure that 
\begin{equation}\label{eq:proper}
(x_3 + \tan(\theta)x_1) \circ Z = (x_3 + \tan(\theta)x_1)\circ H\enskip \text{everywhere on $V$},
\end{equation}
 which is possible by Theorem \ref{th:Mergelyan}-{\rm (II)} up to suitably rotating $H$; cf. \cite[Subsec.\ 5.1]{AL-proper}. As above, in order to construct $\hat H$, we first define it over $\cup_{(i,j)\in I_+} K_j^i$, formally meeting {\rm (4$_Z$)} and \eqref{eq:proper}, and then we extend it to $S_+$ (that is; we define $\hat H$ over $\cup_{(i,j)\in I_+}\I(K_j^i)$) as an $\I$-invariant map.

If $I_-=\emptyset$ the proof is already done; otherwise we consider the $\I$-admissible set 
\[
S_-:=M \cup \big( \cup_{(i,j)\in I_+} \Omega_j^i\cup \I(\Omega_j^i) \big)\cup \big( \cup_{(i,j)\in I_-} K_j^i\cup \I(K_j^i) \big).
\]

To finish the proof, we construct an $\I$-invariant conformal minimal immersion $Y\in\Mcal_\I(V)$ satisfying the following properties:
\begin{enumerate}[\rm (1$_Y$)]
\item $\|Y-Z\|_{1,M \cup ( \cup_{(i,j)\in I_+} \Omega_j^i\cup \I(\Omega_j^i))}<\epsilon/3,$
\item $Y(\overline{\Omega_j^i\setminus K_j^i})\subset \Pi_\delta(\pm\theta)$ for all $(i,j)\in I_\pm$, 
\item $Y(\overline{\beta_j^i\setminus K_j^i})\subset \Pi_{\delta+1}(\pm\theta)$ for all $(i,j)\in I_\pm$,
\item $Y(K_j^i)\subset \Pi_{\delta+1}(\pm\theta)$ for all $(i,j)\in I\setminus I_\pm$, and
\item $\pgot_Y=\pgot$.
\end{enumerate}
Such $Y$ is furnished by Theorem \ref{th:Mergelyan}-{\rm (II)} applied to a suitable $\I$-invariant extension of $Z|_{M \cup (\cup_{(i,j)\in I_+} \Omega_j^i\cup \I(\Omega_j^i))}$ to $S_-$, in a symmetric way to the previous step; cf. again \cite[Subsec.\ 5.1]{AL-proper}.

The immersion $Y$ meets all the requirements in Lemma \ref{lem:fun}.
\end{proof}

\begin{proof}[Proof of Lemma \ref{lem:fun} in case $\chi(V\setminus M^\circ)=-2$.] In this case, there exists an analytical Jordan arc $\gamma\subset V^\circ\setminus M^\circ$, attached to $b M$ at its endpoints and otherwise disjoint to $M$, such that $\gamma\cap\I(\gamma)=\emptyset$, $S:=M\cup\gamma\cup\I(\gamma)$ is an $\I$-admissible set in $\Ncal$, and $\chi(V^\circ \setminus S)=0$; see Remark \ref{rem:topology}-2,3,4. Extend $X$ to a suitable generalized marked immersion $\tilde X_\varpi=(\tilde X,\varpi)\in\Mcal_{\ggot,\I}^*(S)$, satisfying 
\[
\tilde X(S\setminus S^\circ) \subset \Pi_\delta(\theta) \cup \Pi_\delta(-\theta)
\] 
and $\pgot_{\tilde X_\varpi}=\pgot|_{\Hcal_1(S,\z)}$. Applying Theorem \ref{th:Mergelyan} to $\tilde X_\varpi$ we then reduce the proof to the case when $\chi(V\setminus M^\circ)=0$, and are done.
\end{proof}

\begin{proof}[Proof of Theorem \ref{th:proper}]
Reasoning as in the proof of Theorem \ref{th:Mergelyan}, we assume without loss of generality that $M$ is connected.
Set $M_0:=M$ and let $\{M_n\}_{n\in\n}$ be an exhaustion of $\Ncal$ by Runge connected $\I$-invariant compact regions such that the Euler characteristic $\chi(M_n^\circ\setminus M_{n-1})\in\{0,-2\}$ for all $n\in\n$; see Remark \ref{rem:topology}. To prove the theorem we follow the argument that shows \cite[Theorem 5.6]{AL-proper}, using Lemma \ref{lem:fun} instead of \cite[Lemma 5.1]{AL-proper}.
\end{proof}


\subsection{Non-orientable minimal surfaces in $\r^3$ and harmonic functions}\label{subsec:coordenada}

Let $\Ncal$ be an open Riemann surface endowed with an antiholomorphic involution 
$\I\colon\Ncal\to\Ncal$ without fixed points. In this subsection we show that every non-constant $\I$-invariant harmonic function $h\colon \Ncal\to\r$ is a coordinate function of a complete conformal $\I$-invariant minimal immersion $\Ncal\to\r^3$; see Theorem \ref{th:coordenada}. We then derive existence of complete non-orientable minimal surfaces in $\r^3$, with {\em arbitrary conformal structure}, whose Gauss map (see Def.\ \ref{def:gauss-non}) omits one point of the projective plane $\r\p^2$; see Corollary \ref{co:coordenada}. Recall that, by Fujimoto \cite{Fujimoto-points}, the Gauss map of complete non-orientable minimal surfaces in $\r^3$ misses at most two points of $\r\p^2.$ Furthermore, there exist non-orientable Riemann surfaces which do not carry complete conformal minimal immersions into $\r^3$ with Gauss map omitting two points of $\r\p^2$; for instance, by great Picard's theorem, those being parabolic and of finite topology.

The analogous results in the orientable framework were obtained by the authors and Fern\'andez in \cite{AFL-1} (see also \cite{AF-coordenada} for a partial result). Again, we only sketch here the necessary modifications to adapt the proof in \cite{AFL-1} to the non-orientable setting, by using Theorem \ref{th:Mergelyan}. 

\begin{theorem}\label{th:coordenada}
Let $h\colon\Ncal\to\r^3$ be a non-constant $\I$-invariant harmonic function, and let $\pgot\colon\Hcal_1(\Ncal,\z)\to\r^3$ be a group homomorphism such that
$\pgot(\I_*(\gamma))=-\pgot(\gamma)$ and the third coordinate of $\pgot(\gamma)$ equals $\Im \int_\gamma \partial h$, for all  $\gamma\in\Hcal_1(\Ncal,\z).$

Then there exists a complete conformal $\I$-invariant minimal immersion $X=(X_1,X_2,X_3)\colon\Ncal\to\r^3$ with $X_3=h$ and $\pgot_X=\pgot$.
\end{theorem}

The proof of Theorem \ref{th:coordenada} relies on a recursive application of the following.

\begin{lemma} \label{lem:coordenada}
Let $M,$ $V \subset \Ncal$ be two Runge $\I$-invariant compact regions with analytical boundary such that $M \subset V^\circ$ and the Euler characteristic $\chi(V\setminus M^\circ)\in\{-2,0\}$.
 
Let $h\colon V\to\r$ be a non-constant $\I$-invariant harmonic function, let $X=(X_1,X_2,X_3) \in \Mcal_\I(M)$, and let $\pgot\colon \Hcal_1(V,\z)\to\r$ be a group homomorphism, satisfying $X_3=h|_M$, $\pgot_X=\pgot|_{\Hcal_1(M,\z)},$ $\pgot(\I_*(\gamma))=-\pgot(\gamma)$, and the third coordinate of $\pgot(\gamma)$ equals $\Im\int_\gamma \partial h$, for all $\gamma\in\Hcal_1(\Ncal,\z).$
 
Then, for any $P_0\in M$ and $\epsilon>0$,  there exists $Y=(Y_1,Y_2,Y_3)\in \Mcal_\I(V)$ enjoying the following properties:
\begin{enumerate}[\rm (i)]
\item The flux map $\pgot_Y$ of $Y$ equals $\pgot.$ 
\item $\|Y-X\|_{1,M}<\epsilon.$  
\item $Y_3=h$.
\item ${\rm dist}_Y(P_0,bV))>1/\epsilon$, where ${\rm dist}_Y$ denotes the distance on $V$ in the intrinsic metric of the immersion $Y$.
\end{enumerate}
\end{lemma}

\begin{proof}[Proof of Lemma \ref{lem:coordenada} in case $\chi(V\setminus M^\circ)=0$]
As in the proof of Lemma \ref{lem:fun}, write $V\setminus M^\circ=\cup_{j=1}^\jgot (A_j \cup \I(A_j)),$ where $\jgot\in\n$ denotes the number of boundary components of $V$  and $A_1,\I(A_1),\ldots, A_\jgot,\I(A_\jgot)$ are pairwise disjoint compact annuli. 

On the interior of each annuli $A_j$, we define a labyrinth of compact sets $\Kcal_j$ {\em adapted to $dh$}  as that in the proof of \cite[Claim 3.2]{AFL-1} (this follows the spirit of Jorge-Xavier's original construction of a complete minimal surface in a slab of $\r^3$ \cite{JorgeXavier}). Denote by $\Kcal=\cup_{j=1}^\jgot \Kcal_j\cup\I(\Kcal_j)$ and denote by $S\subset \Ncal$ the $\I$-admissible set
\[
S=M\cup\Kcal.
\]

To finish, we reason as in the proof of \cite[Claim 3.2]{AFL-1}. In a first step we extend $X$ to $S$ as an $\I$-invariant conformal minimal immersion $\hat X=(\hat X_1,\hat X_2,\hat X_3)\colon S\to\r^3$, such that $\hat X_3=h|_S$ and whose intrinsic metric is sufficiently large over $\Kcal$. In order to find a suitable $\hat X$ we first argue as in \cite[Claim 3.2]{AFL-1} to extend $X$ to $\cup_{j=1}^\jgot \Kcal_j$, and then we define $\hat X$ over $\cup_{j=1}^\jgot \I(\Kcal_j)$ to be $\I$-invariant. The proof now can be concluded by applying Theorem \ref{th:Mergelyan}-{\rm(II)} to $\hat X$; cf.  again the proof of \cite[Claim 3.2]{AFL-1}.
\end{proof}

\begin{proof}[Proof of Lemma \ref{lem:coordenada} in case $\chi(V\setminus M^\circ)=-2$.] Let $\gamma\subset V^\circ\setminus M^\circ$ be an analytical Jordan arc attached to $b M$ at its endpoints and otherwise disjoint to $M$, such that $\gamma\cap\I(\gamma)=\emptyset$, $S:=M\cup\gamma\cup\I(\gamma)$ is an $\I$-admissible set in $\Ncal$, and $\chi(V^\circ \setminus S)=0$; see Remark \ref{rem:topology}-2,3,4. Extend $X$ to a generalized marked immersion $\tilde X_\varpi=(\tilde X=(\tilde X_1,\tilde X_2,\tilde X_3),\varpi)\in\Mcal_{\ggot,\I}^*(S)$, satisfying 
$\tilde X_3=h|_S$ and $\pgot_{\tilde X_\varpi}=\pgot|_{\Hcal_1(S,\z)}.$
We then reduce the proof to the case when $\chi(V\setminus M^\circ)=0$, by using Theorem \ref{th:Mergelyan}-{\rm(II)}.
\end{proof}

\begin{proof}[Proof of Theorem \ref{th:coordenada}]
Reasoning as in the proof of Theorem \ref{th:Mergelyan}, we assume without loss of generality that $M$ is connected.
Set $M_0:=M$ and let $\{M_n\}_{n\in\n}$ be an exhaustion of $\Ncal$ by Runge connected $\I$-invariant compact regions such that the Euler characteristic $\chi(M_n^\circ\setminus M_{n-1})\in\{0,-2\}$ for all $n\in\n$; see Remark \ref{rem:topology}. To finish we follow the argument in the proof of \cite[Theorem 4.1]{AFL-1}, replacing \cite[Lemma 3.1]{AFL-1} by Lemma \ref{lem:coordenada}.
\end{proof}

\begin{corollary}\label{co:coordenada}
Let $\pgot\colon\Hcal_1(\Ncal,\z)\to\r^3$ be a group homomorphism such that
$\pgot(\I_*(\gamma))=-\pgot(\gamma)$ for all  $\gamma\in\Hcal_1(\Ncal,\z).$

Then there exists a complete conformal $\I$-invariant minimal immersion $X\colon\Ncal\to\r^3$ such that $\pgot_X=\pgot$, and the complex Gauss map of $\sub X\colon\sub\Ncal\to\r^3$ (see Def.\ \ref{def:gauss-non}) omits one point of $\r\p^2.$
\end{corollary}
\begin{proof}
By Theorem \ref{th:GN}, there exists a nowhere-vanishing holomorphic $1$-form $\vartheta$ on $\Ncal$ such that $\I^*\vartheta=\overline{\vartheta}$ and $\int_\gamma\vartheta =\imath \pgot(\gamma)$ for all $\gamma\in\Hcal_1(\Ncal,\z)$. 
Applying Theorem \ref{th:coordenada} to $h:=\Re\int^P \vartheta$, we get a complete conformal $\I$-invariant minimal immersion $X\colon\Ncal\to\r^3$ such that $\pgot_X=\pgot$, and whose complex Gauss map has neither zeros nor poles. Therefore, the Gauss map of $\sub X\colon\sub\Ncal\to\r^3$ omits one point of $\r\p^2$ (see Remark \ref{rem:G}). This concludes the proof.
\end{proof}

In \cite{AFL-2}, the authors and Fern\'andez extended the results in \cite{AFL-1} to minimal surfaces in $\r^n$, $n\geq 3$. The key tool was a Runge-Mergelyan type theorem for minimal surfaces in $\r^n$. In the forthcoming paper \cite{AL-coordenada-non}, we will show the analogous result in the non-orientable framework; this will allow us to generalize the results in this subsection to non-orientable minimal surfaces in $\r^n$.


\subsection*{Acknowledgments} A.\ Alarc\'{o}n is supported by Vicerrectorado de Pol\'{i}tica Cient\'{i}fica e Investigaci\'{o}n de la Universidad de Granada, and is partially supported by MCYT-FEDER grants MTM2007-61775 and MTM2011-22547, Junta de Andaluc\'{i}a Grant P09-FQM-5088, and the grant PYR-2012-3 CEI BioTIC GENIL (CEB09-0010) of the MICINN CEI Program.

F.\ J.\ L\'{o}pez is partially supported by MCYT-FEDER research projects MTM2007-61775 and MTM2011-22547, and Junta de Andaluc\'{\i}a Grant P09-FQM-5088.



\end{document}